\documentclass[a4paper,USenglish]{article}

\usepackage[USenglish]{babel}
\usepackage[ansinew]{inputenc}
\usepackage{amsmath}
\usepackage{enumitem}
\usepackage{amssymb}
\usepackage{amsthm}
\usepackage{adjustbox}
\usepackage{microtype}
\usepackage{xspace}
\usepackage{todonotes}
\usepackage{hyperref}
\usepackage{tikz}
\usepackage{graphicx}
\usepackage{pgfplots}
\usepackage[ruled,vlined]{algorithm2e}
\usepackage{thmtools}
\usepackage{cleveref}
\usepackage{thm-restate}

\usepackage{booktabs}
\usepackage{multirow}
\usepackage{rotating}

\usepackage{authblk}

\Crefname{algocfline}{Algorithm}{Algorithms}
\Crefname{algocf}{Algorithm}{Algorithms}

\pgfplotsset{compat=1.14}

\usetikzlibrary{arrows.meta,arrows}
\usetikzlibrary{patterns}
\usepgfplotslibrary{fillbetween}

\pgfplotsset{every x tick label/.append style={font=\small, yshift=0.0ex}}
\pgfplotsset{every y tick label/.append style={font=\small, yshift=0.0ex}}
\usetikzlibrary{patterns}
\makeatletter
\pgfdeclarepatternformonly[\LineSpace]{my north east lines}{\pgfqpoint{-1pt}{-1pt}}{\pgfqpoint{\LineSpace}{\LineSpace}}{\pgfqpoint{\LineSpace}{\LineSpace}}%
{
    \pgfsetcolor{\tikz@pattern@color}
    \pgfsetlinewidth{0.4pt}
    \pgfpathmoveto{\pgfqpoint{0pt}{0pt}}
    \pgfpathlineto{\pgfqpoint{\LineSpace + 0.1pt}{\LineSpace + 0.1pt}}
    \pgfusepath{stroke}
}
\makeatother
\newdimen\LineSpace
\tikzset{
    line space/.code={\LineSpace=#1},
    line space=10pt
}
\usepgfplotslibrary{fillbetween}

\tikzset{
  schraffiert/.style={pattern=horizontal lines,pattern color=#1},
  schraffiert/.default=black
}
\tikzstyle{densely dashed}=          [dash pattern=on 6pt off 2pt]

\theoremstyle{definition}
\newtheorem{Sat}{Satz}
\numberwithin{Sat}{section}
\newtheorem{theorem}[Sat]{Theorem}
\newtheorem{lemma}[Sat]{Lemma}

\newtheorem{proposition}[Sat]{Proposition}
\newtheorem{definition}[Sat]{Definition}
\newtheorem{observation}[Sat]{Observation}

\def\myoverset#1#2#3{\stackrel{\text{\makebox[#3pt]{#1}}}{#2}}

\renewcommand{\L}[1]{\ell_{#1}}
\newcommand{\TL}[1]{t^L_{#1}}
\newcommand{\TR}[1]{t^R_{#1}}
\newcommand{\SL}[1]{\sigma^L_{#1}}
\newcommand{\SR}[1]{\sigma^R_{#1}}

\newcommand{\TA}{p_0}
\newcommand{\TB}{p_1}
\newcommand{\SA}{\sigma_0}
\newcommand{\SB}{\sigma_1}

\newcommand{\delay}{\text{delay}}

\newcommand{\SLE}[1]{\sigma_{\le #1}}

\newcommand{\alg}{\textsc{Alg}\xspace}
\newcommand{\opt}{\textsc{Opt}\xspace}

\newcommand{\opts}{\textsc{Opt}(\sigma)}

\newcommand{\rh}{{\rho}}
\newcommand{\eps}{\varepsilon}

\newcommand{\sgn}[1]{\text{sgn}(#1)}

\newcommand{\R}{\mathbb{R}}
\newcommand{\N}{\mathbb{N}}

\newcommand{\dar}{\textsc{Dial-a-Ride}\xspace}

\newcommand{\algo}{\textsc{SmarterStart}\xspace}
\newcommand{\smartstart}{\textsc{Smartstart}\xspace}

\newcommand{\algoS}{\algo(\sigma)}
\newcommand{\optS}{\textsc{Opt}(\sigma)}

\newcommand{\timeReq}{r} 
\newcommand{\timeSch}{t} 
\newcommand{\posReq}{a} 
\newcommand{\posSch}{p} 

\newcommand{\pos}[1]{\text{pos}\left(#1\right)} 
\newcommand{\move}[1]{\text{move}(#1)} 

\DeclareMathOperator*{\argmin}{argmin}

\begin{document}
\title{Improved Bounds for Open Online Dial-a-Ride on the Line\footnote{This work was supported by the `Excellence Initiative' of the German Federal and State Governments and the Graduate School~CE at TU~Darmstadt.}}

\author[1]{Alexander Birx}
\author[1]{Yann Disser}
\author[2]{Kevin Schewior\thanks{Supported by the DAAD within the PRIME program using funds of BMBF and the EU Marie Curie Actions.}}
\affil[1]{
	\small Institute of Mathematics and Graduate School CE, TU Darmstadt, Germany.
}
\affil[2]{
	\small Institut f\"ur Informatik, Technische Universit\"at M\"unchen, Garching, Germany.
}

\maketitle
\thispagestyle{empty}
\begin{abstract}
We consider the open, non-preemptive online \dar problem on the real line, where transportation requests appear over time and need to be served by a single server.
We give a lower bound of~$2.0585$ on the competitive ratio, which is the first bound that strictly separates online \dar on the line from online TSP on the line in terms of competitive analysis, and is the best currently known lower bound even for general metric spaces.
On the other hand, we present an algorithm that improves the best known upper bound from~$2.9377$ to~$2.6662$.
The analysis of our algorithm is tight.

\end{abstract}

\section{Introduction}\label{section: Introduction}

We consider the online \dar problem on the line, where transportation requests appear over time and need to be transported to their respective destinations by a single server.
More precisely, each request is of the form $\sigma_i = (a_i,b_i;r_i)$ and appears in position $a_i \in \mathbb{R}$ along the real line at time $r_i \geq 0$ and needs to be transported to position~$b_i \in \mathbb{R}$.
The server starts at the origin, can move at unit speed, and has a capacity~$c \in \mathbb{N} \cup \{\infty\}$ that bounds the number of requests it can carry simultaneously.
The objective is to minimize the completion time, i.e., the time until all requests have been served.
In this paper, we focus on the \emph{non-preemptive} and \emph{open} setting, where the former means that requests can only be unloaded at their destinations, and the latter means that we do not require the server to return to the origin after serving all requests.

We aim to bound the \emph{competitive ratio} of the problem, i.e., the smallest ratio any online algorithm can guarantee between the completion time of its solution compared to an (offline) optimum solution that knows all requests ahead of time.
To date, the best known lower bound of $2.0346$ on this ratio was shown by Bjelde et al.~\cite{Disser1}, already for online TSP, where $a_i = b_i$ for all requests (i.e., requests only need to be visited).
The best known upper bound of~$2.9377$ was achieved by the \smartstart algorithm~\cite{BirxDisser/19}.

\subparagraph*{Our results.}

Our first result is an improved lower bound for online \dar on the line. 
Importantly, since the bound of roughly $2.0346$ was shown to be tight for online TSP~\cite{Disser1}, our new bound is the first time that \dar on the line can be strictly separated from online TSP in terms of competitive analysis.
In addition, our bound is the currently best known lower bound even for general metric spaces.
Specifically, we show the following.

\begin{restatable}{theorem}{theoremGeneralLowerBound}\label{theorem: General Lower Bound}
Let $\rho\approx 2.0585$ be the second largest root of the polynomial $4\rh^3-26\rh^2+39\rh-5$. There is no $(\rho-\eps)$-competitive algorithm for open, non-preemptive ($c<\infty$) online \dar on the line for any $\eps>0$.
\end{restatable}

Our construction is a non-trivial variation of the construction achieving roughly~$2.0346$ for online TSP~\cite{Disser1}.
This construction is comprised of an initial request, a first stage consisting in turn of different iterations, and a second stage. We show that, by using a proper transportation requests as initial requests, we can adapt a single iteration of the first stage as well as the second stage to achieve the bound of roughly~$2.0585$ in the \dar setting.

Our second result is an improved algorithm \algo for online \dar on the line.
This algorithm improves the waiting strategy of the \smartstart algorithm, which was identified as a weakness in~\cite{BirxDisser/19}.
We show that this modification improves the competitive ratio of the algorithm and give a tight analysis.
Specifically, we show the following.

\begin{restatable}{theorem}{theoremMainTheorem}\label{theorem: Main Theorem}
The competitive ratio of $\algo$ is (roughly) $2.6662$.
\end{restatable}

The general idea of \algo is to improve the tradeoff between the case when the algorithm waits before starting its final schedule and the case when it starts the final schedule immediately.
Our modification of \smartstart significantly improves the performance in the former case, while only moderately degrading the performance in the latter case.
Overall, this results in an improved worst-case performance.

\subparagraph*{Related Work.}

The online \dar problem has received considerable attention in the past (e.g.~\cite{Ascheuer1,BirxDisser/19,Disser1, BlomKrumkePaepeStougie/01, Feuerstein1, JailletWagner/08}).
Table~\ref{tab:results} gives an overview of the currently best known bounds on the line for open online \dar and its special case open online TSP.

\begin{table}
\scriptsize
\begin{centering}
\begin{tabular*}{1.0\textwidth}{@{\extracolsep{\fill}}cccccc}
\toprule 
 & & \multicolumn{2}{c}{open} & \multicolumn{2}{c}{closed}\tabularnewline
 & & lower bound & upper bound & lower bound & upper bound\tabularnewline
\midrule
\multirow{3}{*}{
\begin{turn}{90}
\!\tiny line
\end{turn}
} & non-preemptive & $\mathbf{2.0585}$~{\scriptsize(Thm~\ref{theorem: General Lower Bound})} & $\mathbf{2.6662}$~{\scriptsize(Thm~\ref{theorem: Main Theorem})} & $\mathbf{1.75}$~\cite{Disser1} & $2$\tabularnewline
& preemptive & $2.04$ & $\mathbf{2.41}$~\cite{Disser1} & $1.64$ & $2$\tabularnewline
& TSP & $\mathbf{2.04}$~\cite{Disser1} & $\mathbf{2.04}$~\cite{Disser1} & $\mathbf{1.64}$~\cite{Ausiello1} & $\mathbf{1.64}$~\cite{Disser1}\tabularnewline
\midrule
\multirow{3}{*}{
\begin{turn}{90}
\!\tiny general
\end{turn}
} & non-preemptive & $2.0585$~{\scriptsize(Thm~\ref{theorem: General Lower Bound})} & $\mathbf{3.41}$~\cite{Krumke0} & $2$ & $\mathbf{2}$~\cite{Ascheuer1, Feuerstein1}\tabularnewline
& preemptive & $2.04$ & $3.41$ & $2$ & $2$\tabularnewline
& TSP & $2.04$ & $\mathbf{2.5}$~\cite{Ausiello1} & $\mathbf{2}$~\cite{Ausiello1} & $\mathbf{2}$~\cite{Ausiello1}\tabularnewline
\bottomrule
\end{tabular*}
\par\end{centering}
\medskip
\caption{Overview of the best known bounds for online \dar on the line (top), and online \dar on general metric spaces (bottom). Results are split into the non-preemptive case (with $c < \infty$), the preemptive case, and the TSP-case, where source and destination of each request coincide.
Bold results are original, all other results follow immediately.\label{tab:results}}
\end{table}

The following results are known for \emph{closed} online \dar:
For general metric spaces, the competitive ratio is exactly $2$, both for online \dar as well as online TSP~\cite{Ascheuer1,Ausiello1,Feuerstein1}.
On the line, a better upper bound is known only for online TSP, where the competitive ratio is exactly~$(9+\sqrt{17})/8\approx 1.6404$~\cite{Ausiello1,Disser1}.
The best known lower bound for closed, non-preemptive \dar on the line is $1.75$~\cite{Disser1}.

When the objective is to minimize the maximum flow time, on many metric spaces no online algorithm can be competitive~\cite{Krumke2, Krumke3}.
Hauptmeier et al.~\cite{Hauptmeier1} showed that a competitive algorithm is possible if we restrict ourselves to instances with ``reasonable'' load.
Yi and Tian~\cite{Yi2} considered online \dar with deadlines, where the objective is to maximize the number of requests that are served in time.
Other interesting variants of online \dar where destinations of requests are only revealed upon their collection were studied by Lipmann et al.~\cite{Lipmann1} as well as Yi and Tian~\cite{Yi1}.

For an overview of results for the offline version of \dar on the line, see~\cite{Paepe1}.
Without release times, Gilmore and Gomory~\cite{GilmoreGomory/64} and Atallah and Kosaraju~\cite{Atallah1} gave a polynomial time algorithm for closed, non-preemptive \dar on the line with capacity~$c=1$.
Guan~\cite{Guan1} showed that the closed, non-preemptive problem is hard for~$c=2$, and Bjelde et al.~\cite{Disser1} extended this result for any finite capacity $c \geq 2$ in both the open and the closed variant.
Bjelde et al.~\cite{Disser1} also showed that the problem with release times is already hard for finite~$c \geq 1$ in both variants, and Krumke~\cite{Krumke0} gave a 3-approximation algorithm for the closed variant.
The complexity for the case~$c = \infty$ remains open.
For closed, preemptive \dar on the line without release times, Atallah and Kosaraju~\cite{Atallah1} gave a polynomial time algorithm for~$c = 1$ and Guan~\cite{Guan1} for~$c \geq 2$.
Charikar and Raghavachari~\cite{Charikar1} presented approximation algorithms for the closed case without release times on general metric spaces.

\section{General Lower Bound}\label{section: General Lower Bound}

In this section, we prove Theorem~\ref{theorem: General Lower Bound}. Let $c<\infty$ and \alg be a deterministic online algorithm for open online \dar. Let $\rh\approx 2.0585$, be the second largest root of the polynomial $4\rho^3-26\rho^2+39\rho-5$. We describe a request sequence $\sigma_{\rh}$ such that
$\alg(\sigma_{\rh})\ge \rh\opt(\sigma_{\rh})$.

We first give a high-level description of our construction disregarding many technical details. Our construction is based on that in~\cite{Disser1} for the TSP version of the problem. That construction consists of two \emph{stages}: After an initial request $(1,1;1)$ (assuming w.l.o.g.\ \alg's position at time~$1$ is at most~$0$), the first stage starts. This stage consists of a loop, which ends as soon as two so-called critical requests are established. The second stage consists of augmenting the critical requests by suitable additional ones to show the desired competitive ratio. A single iteration of the loop only yields a lower bound of roughly $2.0298$, but as the number of iterations approaches infinity one can show the tight bound of roughly $2.0346$ in the limit.

In the \dar setting, we show a lower bound of roughly $2.0585$ using the same general structure but only a single iteration. Our additional leeway stems from replacing the initial request $(1,1;1)$ with $c$ initial requests of the form $(1,\delta;1)$ where $\delta>1$: At the time when an initial request is loaded, we show that w.l.o.g.\ all $c$ requests are loaded and then proceed as we did when $(1,1;1)$ was served. 
In the new situation, the algorithm has to first deliver the $c$ initial requests to be able to serve additional requests. For the optimum, the two situations however do not differ, because in the new situation there will be an additional request to the right of $\delta$ later anyway. Interestingly, this leeway turns out to be sufficient not only to create critical requests (w.r.t.\ a slightly varied notion of criticality) for a competitive ratio of larger than $2.0298$ but even strictly larger than $2.0346$. The second stage has to be slightly adapted to match the new notion of criticality. It remains unclear how to use multiple iterations in our setting.

We start by making observations that will simplify the exposition. Consider a situation in which the server is fully loaded. 
First note that it is essentially irrelevant whether we assume that the server, without delivering any of the loaded requests, can still serve requests $(a_i,b_i;t_i)$ for which $a_i=b_i$: If it can, we simply move $a_i$ and $b_i$ by $\varepsilon>0$ apart, forbidding the server to serve it before delivering one of the loaded requests first. Therefore, we assume for simplicity that, when fully loaded, the server has to first deliver a request before it can serve any other one. We note that, in our construction, the above idea can be implemented without loss, not even in terms of $\varepsilon$.

The latter discussion also motivates restricting the space of considered algorithms: We call \alg \emph{eager} if it, when fully loaded with requests with identical destinations, immediately delivers these requests without detour. It is clear that we can transform every algorithm $\alg'$ into an eager algorithm $\alg'_\text{eager}$ by letting it deliver the requests right away, waiting until $\alg'$ would have delivered them, and then letting it continue like $\alg'$. Since $\alg'$ cannot collect or serve other requests while being fully loaded, we have $\alg'_\text{eager}(\sigma)\le \alg'(\sigma)$ for every request sequence $\sigma$.

\begin{observation}\label{observation: Eager Algorithms}
Every algorithm for online \dar can be turned into an eager algorithm with the same competitive ratio.
\end{observation}

Thus, we may assume that \alg is eager. 
We now consider the second stage and then design a first stage to match the second stage. 
Suppose we have two requests $\SR{}=(\TR{},\TR{};\TR{})$ and $\SL{}=(-\TL{},-\TL{},\TL{})$ with $\TL{}\leq\TR{}$ to the right and to the left of the origin, respectively. 
We assume that \alg serves $\SR{}$ first at some time $t^*\ge (2\rh-2)\TL{}+(\rh-2)\TR{}$.
Now suppose we could force \alg to serve $\SL{}$ directly after $\SR{}$, even if additional requests are released. 
Then we could just release the request $\SR{*}=(\TR{},\TR{},2\TL{}+\TR{})$ and we would have
\begin{equation*}
\alg(\sigma_{\rh})=t^*+2\TL{}+2\TR{}\ge 2\rh\TL{}+\rh\TR{}=\rh\opt(\sigma_{\rh}),
\end{equation*}
since \opt can serve the three requests in time $2\TL{}+\TR{}$ by serving $\SL{}$ first. In fact, we will show that we can force \alg into this situation (or a worse situation) if the requests $\SR{}=(\TR{},\TR{};\TR{})$ and $\SL{}=(-\TL{},-\TL{},\TL{})$ satisfy the following properties.
To describe the trajectory of a server, we use the notation ``$\move{a}$'' for the tour that moves the server from its current position with unit speed to the point  $a\in\R$. 

\begin{definition}\label{definition: Critical Requests}
We call the last two requests $\SR{}=(\TR{},\TR{};\TR{})$ and $\SL{}=(-\TL{},-\TL{},\TL{})$ of a request sequence with $0<\TL{}\le\TR{}$ \emph{critical} for \alg if the following conditions hold:
\renewcommand{\theenumi}{\roman{enumi}}
\renewcommand{\labelenumi}{(\theenumi)}
\begin{enumerate}
\item Both tours $\move{-\TL{}}\oplus\move{\TR{}}$ and $\move{\TR{}}\oplus\move{-\TL{}}$ serve all requests presented until time $\TR{}$.
\item \alg serves both $\SR{}$ and $\SL{}$ after time $\TR{}$ and \alg's position at time $\TR{}$ lies between $\TR{}$ and~$-\TL{}$.
\item If \alg serves $\SR{}$ before $\SL{}$, it does so no earlier than $\TR{*}:=(2\rh-2)\TL{}+(\rh-2)\TR{}$.
\item If \alg serves $\SL{}$ before $\SR{}$, it does so no earlier than $\TL{*}:=(2\rh-2)\TR{}+(\rh-2)\TL{}$.
\item It holds that $\frac{\TR{}}{\TL{}}\le \frac{4\rh^2-30\rh+50}{-8\rh^2+50\rh-66}.$
\end{enumerate}
\end{definition}

\begin{restatable}{lemma}{lemmaCriticalRequests}\label{lemma: Critical Requests}
If there is a request sequence with two critical requests for \alg, we can release additional requests such that \alg is not $(\rho-\eps)$-competitive on the resulting instance for any $\eps>0$.
\end{restatable}
\Cref{definition: Critical Requests} differs from \cite[Definition 5]{Disser1} only in property (v), which is $\smash{\frac{\TR{}}{\TL{}}\le 2}$ in the original paper. 
\Cref{lemma: Critical Requests} has been proved in \cite[Lemma 6]{Disser1} for request sequences that satisfy the properties of \cite[Definition 5]{Disser1}, however, a careful inspection of the proof of \cite[Lemma~6]{Disser1} shows that the statement of \Cref{lemma: Critical Requests} also holds for request sequences that only satisfy (v) instead of $\smash{\frac{\TR{}}{\TL{}}\le 2}$.
For a detailed proof, see \Cref{appendix: Proof Of Critical Requests}.
Thus, our goal is to construct a request sequence~$\sigma_\rh$ that 
satisfies all properties of \Cref{definition: Critical Requests}.

The remaining part of this section focusses on establishing critical requests. There are no requests released until time $1$. Without loss of generality, we assume that \alg's position at time $1$ is $\pos{1} \le 0$ (the other case is symmetric). 
Here and throughout, we let $\pos{t}$ denote the position of \alg's server at time~$t$.
Now, let
\begin{equation*}
\delta:=\frac{3\rh^2-11}{-3\rh^3+15\rh-4}
\end{equation*}
and let $c$ initial requests $\SR{(j)}=(1,\delta;1)$ with $j\in\{1,\dots,c\}$ appear. These are the only requests appearing in the entire construction with a starting point differing from the destination. We make a basic observation on how \alg has to serve these requests.\footnote{Omitted proofs can be found in the appendix.}

\begin{lemma}\label{lemma: The Time TL}
\alg cannot collect any of the requests $\SR{(j)}$ before time $2$. If \alg collects the requests after time $\rho\delta-(\delta-1)$ or serves $c'<c$ requests before loading the remaining $c-c'$, it is not $(\rh-\eps)$-competitive.
\end{lemma}
\begin{proof}
\alg cannot collect any $\SR{(j)}$ before time $2$ since its position at time $1$ is $\pos{1}\le 0$.
Moreover, \alg is not $(\rh-\eps)$-competitive if it collects one of the requests after time $\rho\delta-(\delta-1)$, since it cannot finish before time $\rho\delta$ and we have
\begin{equation*}
\alg(\{\SR{(j)}\}_{j\in\{1,\dots, c\}})\ge\rho\delta=\rho\opt(\{\SR{(j)}\}_{j\in\{1,\dots, c\}}).
\end{equation*}
Assume \alg serves $c'<c$ requests before loading the remaining $c-c'$. Then, because of
\begin{equation}\label{equation: Not All Requests Taken}
\delta=\frac{3\rh^2-11}{-3\rh^3+15\rh-4}\overset{\rh>2.056}{>}\frac{2}{3-\rh},
\end{equation}
we have 
\begin{equation*}
\alg(\{\sigma_R^0\})\ge\delta+2(\delta-1)\overset{(\ref{equation: Not All Requests Taken})}{>}\rh\delta =\rh\opt(\{\sigma_R^0\}).
\end{equation*}
\end{proof}

We hence may assume that \alg loads all $c$ requests $\SR{(j)}$ at the same time. Let $\TL{}\in[2,\rho\delta-(\delta-1))$ be the time \alg loads the $c$ requests $\SR{(j)}$. We start the first stage and present a variant of a single iteration of the construction in~\cite{Disser1}: We let the request $\SL{}=(-\TL{},-\TL{};\TL{})$ appear and define the function
\begin{equation*}
\L{}(t)=(4-\rho)\cdot t- (2\rho-2)\cdot \TL{},
\end{equation*}
which can be viewed as a line in the path-time diagram. Because of $\rh>2$, we have $\L{}(\TL{})=(6-3\rh)\TL{}<0<\pos{\TL{}}$, i.e., \alg's position at time $\TL{}$ is to the right of the line $\L{}$.
Thus, \alg crosses the line $\L{}$ before it serves $\SL{}$. Let $\TR{}$ be the time \alg crosses $\L{}$ for the first time and let the request $\SR{}=(\TR{},\TR{};\TR{})$ appear.
Assume \alg crosses the line $\L{}$ and serves $\SR{}$ before $\SL{}$. Then it does not serve $\SR{}$ before time
\begin{equation}
\TR{}+|\L{}(\TR{})-\TR{}|=(2\rho-2)\TL{}+(\rho-2)\TR{}=\TR{*}.\label{equation: Too Late Right}
\end{equation}
Now assume \alg crosses $\L{}$ at time $\TR{}\ge\frac{3\rh-5}{7-3\rh}\TL{}$ and serves $\SL{}$ before $\SR{}$. Then it does not serve serve~$\SL{}$ before time
\begin{align}
\TR{}+|\L{}(\TR{})-(-\TL{})|&=(5-\rho)\TR{}- (2\rh-3)\TL{}\nonumber\\
&\ge(2\rh-2)\TR{}+(7-3\rh)\frac{3\rh-5}{7-3\rh}\TL{}-(2\rh-3)\TL{}\nonumber\\
&=(2\rho-2)\TR{}+(\rho-2)\TL{}=\TL{*}.\label{equation: Too Late Left}
\end{align}
The following lemma shows that the two requests cannot be served before these respective times by establishing that indeed $\TR{}\ge\frac{3\rh-5}{7-3\rh}\TL{}$.

\begin{lemma}\label{lemma: Critital Three And Four}
\alg can neither serve $\SL{}$ before time $\TL{*}$ nor can it serve $\SR{}$ before time $\TR{*}$.
\end{lemma}
\begin{proof}
Since \alg is eager, it delivers the $c$ requests $\SR{(j)}$ without waiting or detour, i.e., we have $\pos{\TL{}+(\delta-1)}=\delta$. 
Furthermore, we have
\begin{align*}
\L{}(\TL{}+(\delta-1))&\myoverset{}{=}{30}(4-\rh) (\TL{}+(\delta-1))- (2\rh-2)\TL{}\\
&\myoverset{}{=}{30}(6-3\rh)\TL{}+(4-\rh)(\delta-1)\\
&\myoverset{}{\le}{30}(6-3\rh)(\rh\delta-(\delta-1))+(4-\rh)(\delta-1)\\
&\myoverset{}{=}{30}\frac{3 \rh^4 - 18 \rh^3 + 3 \rh^2 + 50 \rh - 14}{3 \rh^3 - 15 \rh + 4}\\
&\myoverset{$\rh<2.06$}{<}{30} \delta=\pos{\TL{}+(\delta-1)},
\end{align*}
i.e., \alg's position at time $\TL{}+(\delta-1)$ is to the right of $\ell$. 
The earliest possible time \alg crosses $\L{}$ is the solution of
\begin{equation*}
\L{}(\TR{})=(4-\rh) \TR{}- (2\rh-2)\TL{}=\pos{\TL{}+(\delta-1)}+\TL{}+(\delta-1)-\TR{},
\end{equation*}
which is $\TR{}=\frac{2\rh-1}{5-\rh}\TL{}+\frac{2\delta-1}{5-\rh}$.
The inequality
\begin{align*}
\left(\frac{3\rh-5}{7-3\rh}-\frac{2\rh-1}{5-\rh}\right)\TL{}&=\frac{3\rh^2+3\rh-18}{3\rh^2-22\rh+35}\TL{}\\
&\le\frac{3\rh^2+3\rh-18}{3\rh^2-22\rh+35}(\rh\delta-(\delta-1)))\\
&=\frac{3 \rh^3+ 6 \rh^2- 15 \rh -18}{3\rh^4-15\rh^3-15\rh^2+79\rh-20}\\
&=\frac{2\delta-1}{5-\rh},
\end{align*}
implies that we have 
\begin{equation}
\TR{}\ge\frac{3\rh-5}{7-3\rh}\TL{}.\label{equation: Too Late Line}
\end{equation}
Because of inequality (\ref{equation: Too Late Right}) \alg does not serve~$\SR{}$ before~$\TR{*}$ and because of the inequalities (\ref{equation: Too Late Line}) and (\ref{equation: Too Late Left}) it does not serve~$\SL{}$ before time~$\TL{*}$.
\end{proof}

In fact, also the other properties of critical requests are satisfied.

\begin{lemma}\label{lemma: Indeed Critical}
The requests $\SR{}$ and $\SL{}$ of the request sequence $\sigma_\rh$ are critical.
\end{lemma}
\begin{proof}
We have to show that the requests $\SR{}$ and $\SL{}$ of the request sequence $\sigma_\rh$ satisfy the properties (i) to (v) of \Cref{definition: Critical Requests}.
The release time of every request is equal to its starting position, thus every request can be served/loaded immediately once its starting position is visited and (i) of \Cref{definition: Critical Requests} is satisfied.
At time $\TR{}$ \alg has not served $\SR{}$, because for that it would have needed to go right from time $0$ on; it has not served $\SL{}$ either, because during the period of time $[t_L,t_R]$ \alg and $\SL{}$ were on different sides of $\ell$. This establishes the first part of (ii) of \Cref{definition: Critical Requests}. Furthermore at time $\TR{}$ \alg is at position $\pos{\TR{}}=(4-\rho)\TR{}- (2\rho-2)\TL{}$ with
\begin{equation*}
-\TL{}\le (4-\rho)\TR{}- (2\rho-2)\TL{}\le \TR{}
\end{equation*}
Therefore, the second part of (ii) of \Cref{definition: Critical Requests} is satisfied as well.

\Cref{lemma: Critital Three And Four} shows that (iii) and (iv) of \Cref{definition: Critical Requests} are satisfied.
It remains to show that property (v) is satisfied. For this we need to examine the release time $\TR{}$ of $\SR{}$. 
The time $\TR{}$ is largest if \alg tries to avoid crossing the line $\L{}$ for as long as possible, i.e., it continues to move right after serving the requests $\SR{(j)}$. Then, we have $\pos{t}=1-\TL{}+t$ for $t\in[\TL{},\TR{}]$ and $\TR{}$ is the solution of
\begin{equation*}
1-\TL{}+\TR{}=(4-\rh)\TR{}-(2\rh-2)\TL{}.
\end{equation*}
Thus, in general, we have $\TR{}\le\frac{2\rh-3}{3-\rh}\TL{}+\frac{1}{3-\rh}$, i.e.,
\begin{equation}
\frac{\TR{}}{\TL{}}\le\frac{2\rh-3}{3-\rh}+\frac{1}{(3-\rh)\TL{}}\overset{\TL{}\ge 2}{\le}\frac{4\rh-5}{6-2\rh}.\label{equation: Actual Ratio TL And TR}
\end{equation}
For property (v), we need $\frac{\TR{}}{\TL{}}\le \frac{4\rh^2-30\rh+50}{-8\rh^2+50\rh-66}$.
This is satisfied if
\begin{equation*}
\frac{4\rh-5}{6-2\rh}\le \frac{4\rh^2-30\rh+50}{-8\rh^2+50\rh-66},
\end{equation*}
which is equivalent to
\begin{equation*}
4\rh^3-26\rh^2+39\rh-5\ge 0,
\end{equation*}
which is true by definition of $\rho$.\end{proof}

Together with \Cref{lemma: Critical Requests}, this completes the proof of Theorem~\ref{theorem: General Lower Bound}.

\section{An Improved Algorithm}\label{section: An Improved Algorithm}

One of the simplest approaches for an online algorithm to solve \dar is the following: 
Always serve the set of currently unserved requests in an optimum offline schedule and ignore all new incoming request while doing so. 
Afterwards, repeat this procedure with all ignored unserved requests until no new requests arrive. 
This simple algorithm that is often called \textsc{Ignore}~\cite{Ascheuer1} has a competitive ratio of exactly $4$~\cite{BirxDisser/19,Krumke0}. 
The main weakness of \textsc{Ignore} is that it always starts its schedule immediately. 
Ascheuer et al. showed that it is beneficial if the server waits sometimes before starting a schedule and introduced the \smartstart algorithm~\cite{Ascheuer1}, which has a competitive ratio of roughly $2.94$~\cite{BirxDisser/19}.

We define~$L(\timeSch,\posSch,R)$ to be the smallest makespan of a schedule that starts at position~$p$ at time~$t$ and serves all requests in $R \subseteq \sigma$ after they appeared (i.e., the schedule must respect release times).
For the description of online algorithms, we denote by $t$ the current time and by $R_t$ the set of requests that have appeared until time~$t$ but have not been served yet. 

\vspace*{1em}
\begin{algorithm}
\SetKwBlock{Repeat}{repeat}{}
\DontPrintSemicolon
$\posSch_1\leftarrow 0$\;

\For{$j = 1,2,\dots$}{
\While{current time $t < L(\timeSch,\posSch_j,R_t)/(\Theta-1)$}{wait\;}
$\timeSch_j\leftarrow t$\;
$S_j\leftarrow$ optimal offline schedule serving $R_t$ starting from $\posSch_j$\;
execute $S_j$\;
$\posSch_{j+1}\leftarrow $ current position}
 \caption{$\smartstart$}\label{algorithm: Smartstart}
\end{algorithm}
\vspace*{1em}
The algorithm \smartstart is given in \Cref{algorithm: Smartstart}. 
Essentially, at time~$t$, \smartstart waits before starting an optimal schedule to serve all available requests at time
\begin{equation}
  \min_{\timeSch' \ge \timeSch}\left\{\timeSch' \ge\frac{L(\timeSch',\posSch,R_{\timeSch'})}{\Theta-1}\right\},\label{equation: Smartstart Definition}
\end{equation}
where $\posSch$ is the current position of the server and $\Theta > 1$ is a parameter of the algorithm that scales the waiting time.
Importantly, like \textsc{Ignore}, \smartstart ignores incoming requests while executing a schedule. 

Birx and Disser identified that \smartstart's waiting routine defined by inequality (\ref{equation: Smartstart Definition}) has a critical weakness \cite[Lemma 4.1]{BirxDisser/19}.
It is possible to lure the server to any position $q$ in time $q+\eps$ for every $\eps>0$.
Roughly speaking, a request $\sigma_1=((\Theta-1)\eps,(\Theta-1)\eps;(\Theta-1)\eps)$ is released first and then for every $i\in\{2,\dots,\frac{q}{\eps}\}$ a request $\sigma_i=(i\eps,i\eps;i\eps)$ follows.
The schedule to serve the request $\sigma_1$ is started at time $\eps$ and finished at time $2\eps$. The schedule to serve the request at position $i\eps$ is not started earlier than time 
\begin{equation}
  \frac{L(i\eps,(i-1)\eps,\{\sigma_i\})}{\Theta-1}=\frac{|(i-1)\eps-i\eps|}{\Theta-1}=\frac{\eps}{\Theta-1}.
\end{equation}
This time is (depending on the choice of $\Theta$) later than the current time $i\eps$ for every $i\ge 2$. 
Thus there is no waiting time for any schedule except the first one and the server reaches position $q$ at time $q+\eps$.
We see that the request sequence to lure the server away heavily uses that inequality (\ref{equation: Smartstart Definition}) relies on \smartstart's current position $\posSch$, when computing the waiting time.
Thus, we modify the waiting routine of \smartstart to avoid luring accordingly.
Denote by $\SLE{t}$ the set of requests that have been released until time $t$.

\vspace*{1em}
\begin{algorithm}
\SetKwBlock{Repeat}{repeat}{}
\DontPrintSemicolon
$\posSch_1\leftarrow 0$\;

\For{$j = 1,2,\dots$}{
\While{current time $t < L(\timeSch,0,\SLE{\timeSch})/(\Theta-1)$}{wait\;}
$\timeSch_j\leftarrow t$\;
$S_j\leftarrow$ optimal offline schedule serving $R_t$ starting from $\posSch_j$\;
execute $S_j$\;
$\posSch_{j+1}\leftarrow $ current position}

 \caption{\algo}\label{algorithm: Algo}
\end{algorithm}
\vspace*{1em}
The improved algorithm \algo is given in \Cref{algorithm: Algo}. 
At time~$t$, it waits before starting an optimal schedule to serve all available requests at time
\begin{equation}
  \min_{\timeSch' \ge \timeSch}\left\{\timeSch' \ge\frac{L(\timeSch',0,\SLE{\timeSch'})}{\Theta-1}\right\}.\label{equation: Algo Definition}
\end{equation}
Again, $\Theta > 1$ is a parameter of the algorithm that scales the waiting time.
In contrast to \smartstart, the waiting time is dependent on the length of the optimum offline schedule serving all requests appeared until the current time and starting from the origin.
This guarantees that the server cannot be forced to reach any position $q$ before time $q/(\Theta-1)$ since we always have $L(\timeSch,0,\SLE{\timeSch})>q$ if $\SLE{\timeSch}$ contains a request with destination in position $q$.

Whenever we need to distinguish the behavior of \algo for different values of~$\Theta > 1$, we write $\algo_{\Theta}$ to make the choice of~$\Theta$ explicit.
The length of \algo's trajectory is denoted by $\algo(\sigma)$.
Note that the schedules used by \textsc{Ignore}, \smartstart and \algo are NP-hard to compute for $1 < c < \infty$, see~\cite{Disser1}. 

We let $N\in \mathbb{N}$ be the number of schedules needed by \algo to serve~$\sigma$.
The $j$-th schedule is denoted by $S_j$, its starting time by~$\timeSch_j$, its starting point by~$\posSch_j$, its ending point by~$\posSch_{j+1}$, and the set of requests served in~$S_j$ by~$\sigma_{S_j}$.
For convenience, we set $\timeSch_0 = \posSch_0 = 0$.

\subsection{Upper Bound for {\normalfont\algo}}\label{Upper Bound for smarterstarts competitive ratio}
We show the upper bound of \Cref{theorem: Main Theorem}.
The completion time of \algo is
\begin{equation}
\algoS=\timeSch_N+L(\timeSch_N,\posSch_N,\sigma_{S_N}).\label{equation: Costs Algo}
\end{equation}
First, observe that, for all $0 \leq t \leq t'$, $p, p' \in \mathbb{R}$, and $R \subseteq \sigma$, we have
\begin{align}
  L(\timeSch,\posSch,R) &\ge  L(\timeSch',\posSch,R),           \label{equation: Schedule Time}\\
  L(\timeSch,\posSch,R) &\le  |\posSch - \posSch'| + L(\timeSch,\posSch',R),  \label{equation: Schedule Triangule Eq}\\
  L(t,0,\SLE{t})&\le L(t,0,\sigma)\le L(0,0,\sigma)\le \opts.\label{equation: Schedule Less Than Opt}
\end{align}
Similar to \cite{BirxDisser/19}, we distinguish between two cases, depending on whether or not \algo waits after finishing schedule~$S_{N-1}$ and before starting the final schedule~$S_N$.
If \algo waits, the starting time of schedule~$S_N$ is given by
\begin{equation}
\timeSch_N=\frac{1}{\Theta-1}L(\timeSch_N,0,\SLE{\timeSch_N}),\label{equation: Starting Time No Wait}
\end{equation}
otherwise, we have
\begin{equation}\label{equation: Starting Time Wait}
\timeSch_N=\timeSch_{N-1}+L(\timeSch_{N-1},\posSch_{N-1},\sigma_{S_{N-1}}).
\end{equation}
We start by giving a lower bound on the starting time of a schedule.
It was shown in \cite{BirxDisser/19} that the schedule $S_j$ of \smartstart is never started earlier than time $\frac{|\posSch_{j+1}|}{\Theta}$. 
This changes slightly for \algo.

\begin{lemma}\label{lemma: Lower Bound Starting Time}
Algorithm \algo does not start schedule $S_j$ earlier than time $\frac{|\posSch_{j+1}|}{\Theta-1}$, i.e., we have $\timeSch_j\ge\frac{|\posSch_{j+1}|}{\Theta-1}$.
\end{lemma}
\begin{proof}
Since $\posSch_{j+1}$ is the ending point of schedule $S_j$, there is a request with destination in $\posSch_{j+1}$ in the set $\sigma_{S_j}$. All requests of $\sigma_{S_j}$ appear before time $\timeSch_j$, which implies that they are part of the set $\SLE{\timeSch_j}$. Thus, we have
\begin{equation}
L(\timeSch_j,0,\SLE{\timeSch_j})\ge |\posSch_{j+1}|\label{equation: Estimate Length With Position}
\end{equation}
and therefore
\begin{equation*}
\timeSch_j\overset{(\ref{equation: Algo Definition})}{\ge}\frac{L(\timeSch_j,0,\SLE{\timeSch_j})}{\Theta-1}\overset{(\ref{equation: Estimate Length With Position})}{\ge} \frac{|\posSch_{j+1}|}{\Theta-1}\qedhere
\end{equation*}
\end{proof}

Using \Cref{lemma: Lower Bound Starting Time}, we can give an upper bound on the length of \algo's schedules, which is an essential ingredient in our upper bounds. The following lemma is proved similarly to \cite[Lemma 3.2]{BirxDisser/19}, which yields an upper bound of $(1+\frac{\Theta}{\Theta+2})\optS$ for the length of every schedule $S_j$ of \smartstart.

\begin{lemma}\label{lemma: Costs per Schedule}
For every schedule $S_j$ of \algo, we have
\begin{equation*}
L(\timeSch_j,\posSch_j,\sigma_{S_j})\le \left(1+\frac{\Theta-1}{\Theta+1}\right)\optS.  
\end{equation*}
\end{lemma}
\begin{proof}
First, we notice that by the triangle inequality we have
\begin{equation}\label{equation: Triangle Zero}
L(\timeSch_j,\posSch_j,\sigma_{S_j})\le |\posSch_j|+ L(\timeSch_j,0,\sigma_{S_j})\le \optS+|\posSch_j|.
\end{equation}
Now, let~$\sigma^{\opt}_{S_j}$ be the first request of~$\sigma_{S_j}$ that is picked up by \opt and let~$\posReq^\opt_j$ be its starting position and~$\timeReq^\opt_j$ be its release time. We have
\begin{equation}\label{equation: Triangle First Opt}
L(\timeSch_j,\posSch_j,\sigma_{S_j})\le |\posReq^\opt_j-\posSch_j|+ L(\timeSch_j,\posReq^\opt_j,\sigma_{S_j}),
\end{equation}
again by the triangle inequality. Since \opt serves all requests of $\sigma_{S_j}$ starting at position $\posReq^\opt_j$ no earlier than time $\timeReq^\opt_j$, we have
\begin{equation}\label{equation: Opt Final Schedule}
L(\timeSch_j,\posReq^\opt_j,\sigma_{S_j})\overset{\timeReq^\opt_j\le \timeSch_j}{\le}L(\timeReq^\opt_j,\posReq^\opt_j,\sigma_{S_j})\le \optS-\timeReq^\opt_j,
\end{equation}
which yields
\begin{align}
L(\timeSch_j,\posSch_j,\sigma_{S_j})&\myoverset{(\ref{equation: Triangle First Opt})}{\le}{35}|\posReq^\opt_j-\posSch_j|+ L(\timeSch_j,\posReq^\opt_j,\sigma_{S_j})\nonumber\\
&\myoverset{(\ref{equation: Opt Final Schedule})}{\le}{35}\optS+|\posReq^\opt_j-\posSch_j|-\timeReq^\opt_j\nonumber\\
&\myoverset{$\timeSch_{j-1}<\timeReq^\opt_j$}{<}{35}\optS+|\posReq^\opt_j-\posSch_j|-\timeSch_{j-1}.\label{equation: Schedule Length Via Opt Start}
\end{align}
Since $\posSch_j$ is the destination of a request, \opt needs to visit it. In the case that \opt visits $\posSch_j$ before collecting $\sigma^\opt_{S_j}$, \opt still has to collect and serve every request of $\sigma_{S_j}$ after it has visited position $\posSch_j$ the first time, which directly implies
\begin{equation*}
\left(1+\frac{\Theta-1}{\Theta+1}\right)\optS>\optS\ge L(|\posSch_j|,\posSch_j,\sigma_{S_j})\overset{|\posSch_j|\le \timeSch_j}{\ge} L(\timeSch_j,\posSch_j,\sigma_{S_j}).
\end{equation*}
On the other hand, if \opt collects $\sigma^\opt_{S_j}$ before visiting the position $\posSch_j$, we have
\begin{equation}
\timeSch_{j-1}+|\posReq^\opt_j-\posSch_j|\overset{\timeSch_{j-1}<\timeReq^\opt_j}{<}\timeReq^\opt_j+|\posReq^\opt_j-\posSch_j|\le \optS,\label{equation: Opt Collect Before Visit}
\end{equation}
since \opt cannot collect $\sigma^\opt_{S_j}$ before time $\timeReq^\opt_j$ and then still has to visit position $\posSch_j$. 
Thus, we have
\begin{align}
L(\timeSch_j,\posSch_j,\sigma_{S_j})&\myoverset{(\ref{equation: Schedule Length Via Opt Start})}{<}{35} \optS+|\posReq^\opt_j-\posSch_j|-\timeSch_{j-1}\nonumber\\
&\myoverset{(\ref{equation: Opt Collect Before Visit})}{\le}{35}2\optS-2\timeSch_{j-1}\nonumber\\
&\myoverset{\text{Lem.~}\ref{lemma: Lower Bound Starting Time}}{\le}{35}2\optS-2\frac{|\posSch_j|}{\Theta-1}.\label{equation: Second Approximation For Schedule Length}
\end{align}
This implies
\begin{align*}
L(\timeSch_j,\posSch_j,\sigma_{S_j})&\myoverset{(\ref{equation: Triangle Zero}),(\ref{equation: Second Approximation For Schedule Length})}{\le}{25} \min\left\{\optS+|\posSch_j|,2\optS-\frac{2}{\Theta-1}|\posSch_j|\right\}\\
&\myoverset{}{\le}{25}\left(1+\frac{\Theta-1}{\Theta+1}\right)\optS,
\end{align*}
since the minimum above is largest for $|\posSch_j|=\frac{\Theta-1}{\Theta+1}\optS$.
\end{proof}

The following proposition uses \Cref{lemma: Costs per Schedule} to provide an upper bound for the competitive ratio of \algo, in the case that \algo does have a waiting period before starting the final schedule.

\begin{proposition}\label{proposition: Upper Bound Waiting}
In case \algo waits before executing $S_N$, we have
\begin{equation*}
\frac{\algo(\sigma)}{\opt(\sigma)} \le f_1(\Theta) := \frac{2\Theta^2-\Theta+1}{\Theta^2-1}.
\end{equation*}
\end{proposition}
\begin{proof}
Assume \algo waits before starting the final schedule.
Then \Cref{lemma: Costs per Schedule} yields the claimed bound:
\begin{align*}
\algoS &\myoverset{(\ref{equation: Costs Algo})}{=}{35} \timeSch_N + L(\timeSch_N,\posSch_N,\sigma_{S_N})\\
&\myoverset{(\ref{equation: Starting Time No Wait})}{=}{35} \frac{1}{\Theta-1}L(\timeSch_N,0,\SLE{\timeSch_N})+L(\timeSch_N,\posSch_N,\sigma_{S_N})\\
&\myoverset{(\ref{equation: Schedule Less Than Opt})}{\leq}{35} \frac{1}{\Theta-1}\opts+L(\timeSch_N,\posSch_N,\sigma_{S_N})\\
&\myoverset{Lem.~\ref{lemma: Costs per Schedule}}{\le}{35} \left(\frac{1}{\Theta-1}+1+\frac{\Theta-1}{\Theta+1}\right)\optS\\
&\myoverset{}{=}{35}\frac{2\Theta^2-\Theta+1}{\Theta^2-1}\optS.\qedhere
\end{align*}
\end{proof}

In comparison, the upper bound for the competitive ratio of \smartstart, in case \smartstart has a waiting period before starting the final schedule is $\frac{2\Theta^2+2\Theta}{\Theta^2+\Theta-2}\optS$ \cite[Proposition~3.2]{BirxDisser/19}.
Note that \algo's bound is better than \smartstart's bound for $\Theta>1$.

It remains to examine the case that the algorithm \algo has no waiting period before starting the final schedule. For this we use two lemmas from \cite{BirxDisser/19} originally proved for \smartstart, which are still valid for $\algo$ since they give bounds on the optimum offline schedules independently of the waiting routine.

By~$x_- := \min\{0,\min_{i=1,\dots,n} \{a_i\},\min_{i=1,\dots,n} \{b_i\}\}$ we denote the leftmost position and by $x_+ := \max \{ 0, \max_{i=1,\dots,n} \{a_i\}, \max_{i=1,\dots,n} \{b_i\}\}$ the rightmost position that needs to be visited by the server.
We denote by $y_-^{S_j}$ the leftmost and by $y_+^{S_j}$ the rightmost position that occurs in the requests $\sigma_{S_j}$. Note that $\smash{y_-^{S_j}}$ and $\smash{y_+^{S_j}}$ need not lie on different sides of the origin, in contrast to $x_{-/+}$.

\begin{lemma}[Lemma 3.4, Full Version of \cite{BirxDisser/19}]\label{lemma: Approx Schedule From Zero}
Let $S_j$ with $j\in\{1,\dots,N\}$ be a schedule of \algo. Moreover, let $\optS=|x_-|+x_++y$ for some $y\ge 0$. Then, we have
\begin{equation*}
L(\timeSch_j,0,\sigma_{S_j})\le |\min\{0,y^{S_j}_-\}|+\max\{0,y^{S_j}_+\}+y.
\end{equation*}
\end{lemma}

\begin{lemma}[Lemma 3.6, Full Version of \cite{BirxDisser/19}]\label{lemma: Rightmost Position}
Let $S_j$ with $j\in\{1,\dots,N\}$ be a schedule of \algo. Moreover, let $|x_-|\le x_+$ and $\optS=|x_-|+x_++y$ for some $y\ge 0$. Then, for every point~$p$ that is visited by $S_j$ we have
\begin{equation*}
p\le |\posSch_j|+|\posSch_j-\posSch_{j+1}|+y-|\min\{0,y^{S_j}_-\}|.
\end{equation*}
\end{lemma}

Using the bounds established by \Cref{lemma: Approx Schedule From Zero} and \Cref{lemma: Rightmost Position}, we can give an upper bound for the competitive ratio of \algo if the server is not waiting before starting the final schedule.

\begin{proposition}\label{proposition: Upper Bound No Waiting}
If \algo does not wait before executing~$S_{N}$, we have
\begin{equation*}
\frac{\algo(\sigma)}{\opt(\sigma)} \le f_2(\Theta) := \frac{3\Theta^2+3}{2\Theta+1}.
\end{equation*}
\end{proposition}
\begin{proof}
Assume algorithm \algo does not have a waiting period before the last schedule, i.e., \algo starts the final schedule $S_N$ immediately after finishing $S_{N-1}$. Without loss of generality, we assume $|x_-|\le x_+$ throughout the entire proof by symmetry.

First of all, we notice that we may assume that \algo executes at least two schedules in this case. Otherwise either the only schedule has length $0$, which would imply~$\optS=\algoS=0$, or the only schedule would have a positive length, implying a waiting period. Let~$\sigma^{\opt}_{S_N}$ be the first request of~$\sigma_{S_N}$ that is served by \opt and let~$\posReq^\opt_N$ be its starting point and~$\timeReq^\opt_N$ be its release time. We have
\begin{align}
\algoS &\myoverset{(\ref{equation: Costs Algo})}{=}{30}\timeSch_N+L(\timeSch_N,\posSch_N,\sigma_{S_N})\nonumber\\
&\myoverset{(\ref{equation: Starting Time Wait})}{=}{30}\timeSch_{N-1}+L(\timeSch_{N-1},\posSch_{N-1},\sigma_{S_{N-1}})+L(\timeSch_N,\posSch_N,\sigma_{S_N})\nonumber\\
&\myoverset{$\timeSch_N\ge \timeReq^\opt_N$}{\le}{30} \timeSch_{N-1}+L(\timeSch_{N-1},\posSch_{N-1},\sigma_{S_{N-1}})+L(\timeReq^\opt_N,\posSch_N,\sigma_{S_N}).\label{equation: First Approx Upper Bound No Wait}
\end{align}
Since \opt serves all requests of $\sigma_{S_N}$ after time $\timeReq^\opt_N$, starting with a request with starting point~$\posReq^\opt_N$, we also have
\begin{equation}
\optS\ge \timeReq^\opt_N+L(\timeReq^\opt_N,\posReq^\opt_N,\sigma_{S_N}).\label{equation: First Approx Opt No Wait}
\end{equation}
Furthermore, we have
\begin{equation}
\timeReq^\opt_N>\timeSch_{N-1}\label{equation: Time First Request Last Schedule Opt}
\end{equation}
since otherwise $\sigma^{\opt}_{S_N}\in\sigma_{S_{N-1}}$ would hold. This gives us
\begin{align}
\algoS &\myoverset{(\ref{equation: First Approx Upper Bound No Wait})}{\le}{24}\timeSch_{N-1}+L(\timeSch_{N-1},\posSch_{N-1},\sigma_{S_{N-1}})+L(\timeReq^\opt_N,\posSch_N,\sigma_{S_N})\nonumber\\
&\myoverset{(\ref{equation: Schedule Triangule Eq})}{\le}{24}\timeSch_{N-1}+L(\timeSch_{N-1},\posSch_{N-1},\sigma_{S_{N-1}})\nonumber+|\posReq^\opt_N-\posSch_N|\nonumber\\
&\qquad\qquad+L(\timeReq^\opt_N,\posReq^\opt_N,\sigma_{S_N})\nonumber\\
&\myoverset{(\ref{equation: First Approx Opt No Wait})}{\le}{24}\timeSch_{N-1}+L(\timeSch_{N-1},\posSch_{N-1},\sigma_{S_{N-1}})+|\posReq^\opt_N-\posSch_N|\nonumber\\
&\qquad\qquad+\optS-\timeReq^\opt_N\nonumber\\
&\myoverset{(\ref{equation: Time First Request Last Schedule Opt})}{<}{24}L(\timeSch_{N-1},\posSch_{N-1},\sigma_{S_{N-1}})+|\posReq^\opt_N-\posSch_N|+\optS\label{equation: Kevin}\\
&\myoverset{(\ref{equation: Schedule Triangule Eq})}{\le}{24}|\posSch_{N-1}|+L(\timeSch_{N-1},0,\sigma_{S_{N-1}})+|\posReq^\opt_N-\posSch_N|+\optS\nonumber\\
&\myoverset{\text{Lem.~}\ref{lemma: Lower Bound Starting Time}}{\le}{24}(\Theta-1)\timeSch_{N-2}+L(\timeSch_{N-1},0,\sigma_{S_{N-1}})\nonumber\\
&\qquad\qquad+|\posReq^\opt_N-\posSch_N|+\optS.\label{equation: Second Approx Upper Bound No Wait}
\end{align}
We have
\begin{equation}\label{equation: Approx Opt No Wait Both Cases}
\optS\ge \timeSch_{N-2}+|\posReq^\opt_N-\posSch_N|,
\end{equation}
because \opt has to visit both $\posReq^\opt_N$ and $\posSch_N$ after time $\timeSch_{N-2}$: It has to visit $\posReq^\opt_N$ to collect $\sigma^{\opt}_{S_N}$ and it has to visit $\posSch_N$ to deliver some request of $\sigma_{S_{N-1}}$.
Using the above inequalitiy, we get
\begin{align}
\algoS &\myoverset{(\ref{equation: Second Approx Upper Bound No Wait})}{<}{15}(\Theta-1)\timeSch_{N-2}+L(\timeSch_{N-1},0,\sigma_{S_{N-1}})\nonumber\\
&\qquad+|\posReq^\opt_N-\posSch_N|+\optS\nonumber\\
&\myoverset{(\ref{equation: Approx Opt No Wait Both Cases})}{\le}{15}2\opts+L(\timeSch_{N-1},0,\sigma_{S_{N-1}})+(\Theta-2)\timeSch_{N-2}\label{equation: Bound One Of Two}.
\end{align}
In the case $\Theta\ge 2$, we have
\begin{align*}
\algoS &\myoverset{(\ref{equation: Bound One Of Two})}{<}{25}2\opts+L(\timeSch_{N-1},0,\sigma_{S_{N-1}})+(\Theta-2)\timeSch_{N-2}\\
&\myoverset{(\ref{equation: Schedule Less Than Opt})}{\le}{25}(\Theta+1)\opts\\
&\myoverset{$\Theta\ge 2$}{\leq}{25}\frac{3\Theta^2+3}{2\Theta+1}\opts.
\end{align*}
Thus, we may assume $\Theta<2$. Similarly as in inequality (\ref{equation: Bound One Of Two}), we get
\begin{align}
\algoS &\myoverset{(\ref{equation: Second Approx Upper Bound No Wait})}{<}{15}(\Theta-1)\timeSch_{N-2}+L(\timeSch_{N-1},0,\sigma_{S_{N-1}})\nonumber\\
&\qquad+|\posReq^\opt_N-\posSch_N|+\optS\nonumber\\
&\myoverset{(\ref{equation: Approx Opt No Wait Both Cases})}{\le}{15}\Theta\opts+L(\timeSch_{N-1},0,\sigma_{S_{N-1}})+(2-\Theta)|\posReq^\opt_N-\posSch_N|\nonumber\\
&\myoverset{(\ref{equation: Algo Definition})}{\le}{15}\Theta\opts+(\Theta-1)\timeSch_{N-1}+(2-\Theta)|\posReq^\opt_N-\posSch_N|\nonumber\\
&\myoverset{}{\le}{15}(2\Theta-1)\opts+(2-\Theta)|\posReq^\opt_N-\posSch_N|,\label{equation: Bound Two Of Two}
\end{align}
where the last inequality follows, because there exists a request in $\sigma$ with release date later than $t_{N-1}$.
This means the claim is shown if we have
\begin{equation}\label{equation: Goal}
|\posSch_N-\posReq_N^\opt|\leq \optS-\frac{\Theta-1}{2\Theta+1}\optS
\end{equation}
since then we have 
\begin{align*}
\algoS &\myoverset{(\ref{equation: Bound Two Of Two})}{<}{15}(2\Theta-1)\opts+(2-\Theta)|\posReq^\opt_N-\posSch_N|\\
&\myoverset{(\ref{equation: Goal})}{\leq}{15}(2\Theta-1)\opts+(2-\Theta)\left(1-\frac{\Theta-1}{2\Theta+1}\right)\opts\\
&\myoverset{}{=}{15}\frac{3\Theta^2+3}{2\Theta+1}\opts.
\end{align*}
Therefore, we may assume in the following that
\begin{equation}
|\posSch_N-\posReq_N^\opt|>\optS-\frac{\Theta-1}{2\Theta+1}\optS.\label{equation: Distance Of Starting Points Assumption}
\end{equation}
Let $\optS=|x_-|+x_++y$ for some $y\ge 0$. By definition of $x_-$ and $x_+$ we have
\begin{equation}
|\posSch_N-\posReq_N^\opt|+y\le \optS.\label{equation: Distance Of Starting Points Approx One}
\end{equation}
In the case that \opt visits position $\posSch_N$ before it collects $\sigma^{\opt}_{S_N}$, we have
\begin{equation}\label{equation: Distance Of Starting Points Approx Two}
|\posReq^\opt_N-\posSch_N|+|\posSch_N|\le \opts.
\end{equation}
Similarly, if \opt collects $\sigma^{\opt}_{S_N}$ before it visits position $\posSch_N$ for the first time, we have
\begin{align*}
\optS&\myoverset{}{\ge}{25} \timeReq^\opt_N+|\posReq^\opt_N-\posSch_N|\\ 
&\myoverset{(\ref{equation: Time First Request Last Schedule Opt})}{>}{25} \timeSch_{N-1}+|\posReq^\opt_N-\posSch_N|\\
&\myoverset{Lem.~\ref{lemma: Lower Bound Starting Time}}{\ge}{25} \frac{|\posSch_{N}|}{\Theta-1}+|\posReq^\opt_N-\posSch_N|\\
&\myoverset{$\Theta<2$}{\ge}{25} |\posSch_{N}|+|\posReq^\opt_N-\posSch_N|.
\end{align*}
Thus, inequality (\ref{equation: Distance Of Starting Points Approx Two}) holds in general.
To sum it up, we may assume that
\begin{equation}
\max\{y,|\posSch_N|,\timeSch_{N-2}\}\overset{(\ref{equation: Distance Of Starting Points Assumption}),(\ref{equation: Distance Of Starting Points Approx One}),(\ref{equation: Distance Of Starting Points Approx Two}),(\ref{equation: Approx Opt No Wait Both Cases})}<\frac{\Theta-1}{2\Theta+1}\optS\label{equation: Central Approx Upper Bound No Wait}
\end{equation}
holds. In the following, denote by $y^{S_{N-1}}_-$ the leftmost starting or ending point and by $y^{S_{N-1}}_+$ the rightmost starting or ending point of the requests in $\sigma_{S_{N-1}}$. We compute
\begin{align}
\algoS &\myoverset{(\ref{equation: Kevin})}{<}{25} L(\timeSch_{N-1},\posSch_{N-1},\sigma_{S_{N-1}})+|\posSch_N-\posReq_N^\opt|+\optS\nonumber\\
&\myoverset{(\ref{equation: Distance Of Starting Points Approx Two})}{<}{25} L(\timeSch_{N-1},\posSch_{N-1},\sigma_{S_{N-1}})+2\optS-|\posSch_N|\nonumber\\
&\myoverset{(\ref{equation: Schedule Time})}{\le}{25} |\posSch_{N-1}|+L(\timeSch_{N-1},0,\sigma_{S_{N-1}})+2\optS-|\posSch_N|\nonumber\\
&\myoverset{Lem.~\ref{lemma: Lower Bound Starting Time}}{\le}{25} (\Theta-1) \timeSch_{N-2}+L(\timeSch_{N-1},0,\sigma_{S_{N-1}})+2\optS-|\posSch_N|\nonumber\\
&\myoverset{Lem.~\ref{lemma: Approx Schedule From Zero}}{\le}{25} (\Theta-1) \timeSch_{N-2}+\max\{0,|y^{S_{N-1}}_-|\}+\max\{0,y^{S_{N-1}}_+\}\nonumber\\
&\qquad\qquad+y+2\optS-|\posSch_N|.\label{equation: Fourth Approx Upper Bound No Wait}
\end{align}
Obviously, position $y^{S_{N-1}}_+$ is visited by \algo in schedule $S_{N-1}$. Therefore, $y^{S_{N-1}}_+$ is smaller than or equal to the rightmost point that is visited by \algo during schedule~$S_{N-1}$, which gives us
\begin{equation}
y^{S_{N-1}}_+\overset{\text{Lem.~}\ref{lemma: Rightmost Position}}{\le} |\posSch_{N-1}|+|\posSch_{N-1}-\posSch_N|+y-\max\{0,|y^{S_{N-1}}_-|\}.\label{equation: Upper Bound Rightmost Point}
\end{equation}
On the other hand, because of $|x_-|\le x_+$, we have $\optS\ge 2|x_-|+x_+$, which implies $y\ge |x_-|$. By definition of $x_-$ and $y^{S_{N-1}}_-$, we have $|x_-|\ge \max\{0,|y^{S_{N-1}}_-|\}$. This gives us~$y\ge\max\{0,|y^{S_{N-1}}_-|\}$ and
\begin{equation}
0\le |\posSch_{N-1}|+|\posSch_{N-1}-\posSch_N|+y-\max\{0,|y^{S_{N-1}}_-|\}.\label{equation: Approx Distance Starting Points Last Schedules}
\end{equation}
To sum it up, we have
\begin{equation}
\max\{0,y^{S_{N-1}}_+\}\overset{(\ref{equation: Upper Bound Rightmost Point}),(\ref{equation: Approx Distance Starting Points Last Schedules})}{\le} |\posSch_{N-1}|+|\posSch_{N-1}-\posSch_N|+y-\max\{0,|y^{S_{N-1}}_-|\}.\label{equation: Upper Bound Rightmost Point Or Zero}
\end{equation}
The inequality above gives us
\begin{align*}
\algoS &\myoverset{(\ref{equation: Fourth Approx Upper Bound No Wait})}{<}{25} (\Theta-1) \timeSch_{N-2}+\max\{0,|y^{S_{N-1}}_-|\}+\max\{0,y^{S_{N-1}}_+\}\\
&\qquad\qquad+y+2\optS-|\posSch_N|\\
&\myoverset{(\ref{equation: Upper Bound Rightmost Point Or Zero})}{\le}{25} (\Theta-1) \timeSch_{N-2}+|\posSch_{N-1}|+|\posSch_{N-1}-\posSch_N|+2y\\
&\qquad\qquad+2\optS-|\posSch_N|\\
&\myoverset{}{\le}{25} (\Theta-1) \timeSch_{N-2}+|\posSch_{N-1}|+|\posSch_{N-1}|+|\posSch_N|+2y\\
&\qquad\qquad+2\optS-|\posSch_N|\\
&\myoverset{Lem.~\ref{lemma: Lower Bound Starting Time}}{\le}{25} (\Theta-1) \timeSch_{N-2}+2(\Theta-1) \timeSch_{N-2}+2y+2\optS\\
&\myoverset{(\ref{equation: Central Approx Upper Bound No Wait})}{\le}{25} (3\Theta-3)\frac{\Theta-1}{2\Theta+1}\optS+2\frac{\Theta-1}{2\Theta+1}\optS\\
&\qquad\qquad+2\optS\\
&\myoverset{}{=}{25} \frac{3\Theta^2+3}{2\Theta+1}\optS.\qedhere
\end{align*}
\end{proof}

In comparison, the upper bound for the competitive ratio of \smartstart in case it does not have a waiting period before starting the final schedule is $\Theta+1-\frac{\Theta-1}{3\Theta+3}\optS$ \cite[Proposition 3.4]{BirxDisser/19}. Note that \algo's bound is slightly worse than \smartstart's bound for $\Theta>1.47$. However, in combination with the bound of \Cref{proposition: Upper Bound Waiting}, \algo has a better worst-case than \smartstart.

\begin{theorem}\label{theorem: General Upper Bound}
Let $\Theta^*$ be the largest solution of $f_1(\Theta) = f_2(\Theta)$, i.e.,
\begin{equation*}
\frac{3\Theta^{*2}+3}{2\Theta^*+1}=\frac{2\Theta^{*2}-\Theta^*+1}{\Theta^{*2}-1}.
\end{equation*}
Then, $\algo_{\Theta^*}$ is $\rho^*$-competitive with $\rho^* := f_1(\Theta^*) = f_2(\Theta^*) \approx 2.6662$.
\end{theorem}
\begin{proof}
For the case, where \algo does wait before starting the final schedule, we have established the upper bound
\begin{equation*}
\frac{\algoS}{\optS}\le\frac{2\Theta^2-\Theta+1}{\Theta^2-1}=f_1(\Theta)
\end{equation*}
in \Cref{proposition: Upper Bound Waiting} and for the case, where \algo starts the final schedule immediately after the second to final one, we have established the upper bound
\begin{equation*}
\frac{\algoS}{\optS}\le\frac{3\Theta^2+3}{2\Theta+1}=f_2(\Theta)
\end{equation*}
in \Cref{proposition: Upper Bound No Waiting}. Therefore, if it exists,
\begin{equation*}
\Theta^*=\argmin_{\Theta>1}\left\{\max\{f_1(\Theta),f_2(\Theta)\}\right\}
\end{equation*}
 is the parameter for \algo with the smallest upper bound. We note that $f_1$ is strictly decreasing for $\Theta>1$ and that $f_2$ is strictly increasing for $\Theta>1$. Therefore, if an intersection point of $f_1$ and $f_2$ that is larger than $1$ exists, then this is at $\Theta^*$. Indeed, the intersection point exists, which is the largest solution of
\begin{equation*}
\frac{3\Theta^2+3}{2\Theta+1}=\frac{2\Theta^2-\Theta+1}{\Theta^2-1}.
\end{equation*}
The resulting upper bound for the competitive ratio is
\begin{equation*}
\rho^*=f_1(\Theta^*)=f_2(\Theta^*)\approx 2.6662.\qedhere
\end{equation*}
\end{proof}

\subsection{Lower Bound for {\normalfont\algo}}\label{section: Lower Bound for smarterstarts competitive ratio}
We show the lower bound of \Cref{theorem: Main Theorem}.
In this section, we explicitly construct instances that demonstrate that the upper bounds given in the previous section are tight for certain ranges of $\Theta>1$, in particular for~$\Theta = \Theta^*$ (as in \Cref{theorem: General Upper Bound}). 
Further, we show that choices of $\Theta > 1$ different from $\Theta^*$ yield competitive ratios worse than~$\rho^*\approx 2.67$.
Together, this implies that $\rho^*$ is exactly the best possible competitive ratio for \algo.

\begin{proposition}\label{proposition: Lower Bound Waiting}
Let $1<\Theta<2$.
For every sufficiently small $\eps>0$, there is a set of requests~$\sigma$ such that \algo waits before starting the final schedule and such that the inequality
\begin{equation*}
\frac{\algoS}{\optS}\ge\frac{2\Theta^2-\Theta+1}{\Theta^2-1}-\eps
\end{equation*}
holds, i.e., the upper bound established in \Cref{proposition: Upper Bound Waiting} is tight for $\Theta\in(1,2)$.
\end{proposition}
\begin{proof}
Let $\eps>0$ with $\smash{\eps<\frac{\Theta}{\Theta+1}}$ and $\smash{\eps'=\frac{\Theta+1}{2\Theta}\eps}$. 
Let the request
\begin{equation*}
\sigma_1=(1,1;0)
\end{equation*}
appear. For all $t\ge 0$ we have $L(t,0,\{\sigma_1\})=1$. Thus, \algo starts its first schedule~$S_1$ at time $\smash{\timeSch_1=\frac{1}{\Theta-1}}$ and reaches position $\posSch_2=1$ at time $\smash{\frac{\Theta}{\Theta-1}}$.
Next, we let the second and final request
\begin{equation*}
\sigma_2=(-\frac{1}{\Theta-1}+\eps',1;\frac{1}{\Theta-1}+\eps')
\end{equation*}
appear. For $t\ge\frac{\Theta}{\Theta-1}$ we have 
\begin{align*}
L(t,0,\{\sigma_1,\sigma_2\})&=\left|0-\left(-\frac{1}{\Theta-1}+\eps'\right)\right|+\left|\left(-\frac{1}{\Theta-1}+\eps'\right)-1\right|\\
&=\frac{2}{\Theta-1}-2\eps'+1.
\end{align*}
Thus, the second and final schedule $S_2$ is not started before time 
\begin{equation*}
\frac{L\left(\frac{\Theta}{\Theta-1},0,\{\sigma_1,\sigma_2\}\right)}{\Theta-1}=\frac{2}{(\Theta-1)^2}+\frac{1-2\eps'}{\Theta-1}.
\end{equation*}
By assumption, we have $\Theta<2$ and $\eps<\frac{\Theta}{\Theta+1}$, i.e., $\eps'<\frac{1}{2}$, which implies that for the time~$\frac{\Theta}{\Theta-1}$, when \algo reaches position $\posSch_2=1$, the inequality
\begin{equation}
\frac{L(\frac{\Theta}{\Theta-1},0,\{\sigma_1,\sigma_2\})}{\Theta-1}=\frac{2}{(\Theta-1)^2}+\frac{1-2\eps'}{\Theta-1}\overset{\eps'<\frac{1}{2}}{>}\frac{2}{(\Theta-1)^2}\overset{1<\Theta<2}{>}\frac{\Theta}{\Theta-1}\label{equation: Lower Bound Approx Starting Time One}
\end{equation}
holds. (Note that inequality~(\ref{equation: Lower Bound Approx Starting Time One}) also holds for slightly larger $\Theta$ if we let $\eps\rightarrow 0$.) Because of inequality~(\ref{equation: Lower Bound Approx Starting Time One}), \algo has a waiting period and starts the schedule~$S_2$ at time
\begin{equation*}
\timeSch_2=\frac{L\left(\frac{\Theta}{\Theta-1},0,\{\sigma_1,\sigma_2\}\right)}{\Theta-1}=\frac{2}{(\Theta-1)^2}+\frac{1-2\eps'}{\Theta-1}.
\end{equation*}
Serving $\sigma_2$ from position $\posSch_2=1$ takes time 
\begin{align*}
L(\timeSch_2,\posSch_2,\{\sigma_2\})&=\left|1-\left(-\frac{1}{\Theta-1}+\eps'\right)\right|+\left|\left(-\frac{1}{\Theta-1}+\eps'\right)-1\right|\\
&=2+\frac{2}{\Theta-1}-2\eps'.
\end{align*}
To sum it up, we have
\begin{align*}
\algoS&=\timeSch_2+L(\timeSch_2,\posSch_2,\{\sigma_2\})\\
&=\frac{2}{(\Theta-1)^2}+\frac{1-2\eps'}{\Theta-1}+2+\frac{2}{\Theta-1}-2\eps'\\
&=\frac{2\Theta^2-\Theta+1}{(\Theta-1)^2}-2\eps'\frac{\Theta}{\Theta-1}.
\end{align*}
On the other hand, \opt goes from the origin to $-\frac{1}{\Theta-1}+\eps'$ to collect $\sigma_2$ at time $\frac{1}{\Theta-1}+\eps'$ (i.e., it has to wait for $2\eps'$ units of time after it reaches position $-\frac{1}{\Theta-1}+\eps'$). Then \opt goes straight to position $1$ delivering $\sigma_2$ and serving $\sigma_1$. Therefore, we have
\begin{equation*}
\opts=\left|0-\left(-\frac{1}{\Theta-1}+\eps'\right)\right|+2\eps'+\left|\left(-\frac{1}{\Theta-1}+\eps'\right)-1\right|=\frac{\Theta+1}{\Theta-1}.
\end{equation*}
Note, that \opt can do this even if the capacity is $c=1$, since $\sigma_2$ does not need to be carried over position $1$, where $\sigma_1$ appears. Since we have~$\smash{\eps'=\frac{\Theta+1}{2\Theta}\eps}$, we obtain
\begin{equation*}
\frac{\algoS}{\opts}=\frac{2\Theta^2-\Theta+1}{\Theta^2-1}-\frac{2\eps'\Theta}{\Theta+1}=\frac{2\Theta^2-\Theta+1}{\Theta^2-1}-\eps,
\end{equation*}
as claimed.
\end{proof}

\begin{proposition}\label{proposition: Lower Bound No Waiting}
Let $\frac{1}{2}(1+\sqrt{5})\le\Theta\le 2$. For every sufficiently small $\eps>0$ there is a set of requests~$\sigma$ such that \algo immediately starts $S_N$ after $S_{N-1}$ and such that
\begin{equation*}
\frac{\algoS}{\optS}\ge\frac{3\Theta^2+3}{2\Theta+1}-\eps,
\end{equation*}
i.e., the upper bound established in \Cref{proposition: Upper Bound No Waiting} is tight for $\Theta\in[\frac{1}{2}(1+\sqrt{5}),2]\approx[1.6180,2]$.
\end{proposition}
\begin{proof}
Let $\eps>0$ with $\eps<\frac{1}{4}(\frac{5\Theta^2-9\Theta+4}{2\Theta+1})$ (note that $\frac{5\Theta^2-9\Theta+4}{2\Theta+1}>0$ for $\Theta>1$) and $\eps'=\frac{2\Theta+1}{5\Theta^2-9\Theta+4}\eps$. Let the request 
\begin{equation*}
\sigma_1=(1,1;0)
\end{equation*}
appear. For all $t\ge 0$, we have $L(t,0,\{\sigma_1\})=1$. Thus, \algo starts its first schedule $S_1$ at time $\timeSch_1=\frac{1}{\Theta-1}$ and reaches position $\posSch_2=1$ at time $\frac{\Theta}{\Theta-1}$. Next we let two new requests
\begin{align*}
\sigma_2^{(1)}&=\left(2+\frac{1}{\Theta-1}-2\eps',2+\frac{1}{\Theta-1}-2\eps';\frac{1}{\Theta-1}+\eps'\right),\\
\sigma_2^{(2)}&=\left(-\frac{1}{\Theta-1},-\frac{1}{\Theta-1};\frac{1}{\Theta-1}+\eps'\right)
\end{align*}
appear. 
For $t\ge\frac{\Theta}{\Theta-1}$ we have 
\begin{align*}
L(t,0,\{\sigma_1,\sigma_2^{(1)},\sigma_2^{(2)}\})&=\left|0-\left(-\frac{1}{\Theta-1}\right)\right|+\eps'\\
&\qquad+\left|\left(-\frac{1}{\Theta-1}\right)-\left(2+\frac{1}{\Theta-1}-2\eps'\right)\right|\\
&=\frac{3}{\Theta-1}+2-\eps'.
\end{align*}
Thus, the second schedule $S_2$ is not started before time
\begin{equation*}
\frac{L\left(\frac{\Theta}{\Theta-1},0,\{\sigma_1,\sigma_2^{(1)},\sigma_2^{(2)}\}\right)}{\Theta-1}=\frac{3}{(\Theta-1)^2}+\frac{2-\eps'}{\Theta-1}.
\end{equation*}
By assumption, we have $\Theta<2$ a nd $\eps<\frac{1}{4}(\frac{5\Theta^2-9\Theta+4}{2\Theta+1})$, i.e., $\eps'<\frac{1}{4}$, which implies that for the time~$\frac{\Theta}{\Theta-1}$, when \algo reaches position $\posSch_2=1$, the inequality
\begin{equation}
\frac{L(\frac{\Theta}{\Theta-1},0,\{\sigma_1,\sigma_2^{(1)},\sigma_2^{(2)}\})}{\Theta-1}=\frac{3}{(\Theta-1)^2}+\frac{2-\eps'}{\Theta-1}\overset{\eps'<2}{>}\frac{3}{(\Theta-1)^2}\overset{\Theta<2}{>}\frac{\Theta}{\Theta-1}\label{equation: Lower Bound Approx Starting Time Two}
\end{equation}
holds. 
(Note that inequality~(\ref{equation: Lower Bound Approx Starting Time Two}) also holds for slightly larger $\Theta$ if we let $\eps\rightarrow 0$.) Because of inequality~(\ref{equation: Lower Bound Approx Starting Time Two}), \algo has a waiting period and starts the schedule~$S_2$ at time
\begin{equation*}
\timeSch_2=\frac{L\left(\frac{\Theta}{\Theta-1},0,\{\sigma_1,\sigma_2^{(1)},\sigma_2^{(2)}\}\right)}{\Theta-1}=\frac{3}{(\Theta-1)^2}+\frac{2-\eps'}{\Theta-1}.
\end{equation*}
If \algo serves $\sigma_2^{(2)}$ before serving $\sigma_2^{(1)}$ the time it needs is at least
\begin{equation*}
\left|1-\left(-\frac{1}{\Theta-1}\right)\right|+\left|\left(-\frac{1}{\Theta-1}\right)-\left(2+\frac{1}{\Theta-1}-2\eps'\right)\right|=3+\frac{3}{\Theta-1}-2\eps'.
\end{equation*}
The best schedule that serves $\sigma_2^{(2)}$ after serving $\sigma_2^{(1)}$ needs time
\begin{equation*}
\left|1-\left(2+\frac{1}{\Theta-1}-2\eps'\right)\right|+\left|\left(2+\frac{1}{\Theta-1}-2\eps'\right)-\left(-\frac{1}{\Theta-1}\right)\right|=3+\frac{3}{\Theta-1}-4\eps'.
\end{equation*}
Thus, \algo serves $\sigma_2^{(2)}$ after serving $\sigma_2^{(1)}$ and finishes $S_2$ at position $\posSch_3=-\frac{1}{\Theta-1}$ at time
\begin{equation*}
\timeSch_2+L(\timeSch_2,\posSch_2,\{\sigma_2^{(1)},\sigma_2^{(2)}\})=\frac{3}{(\Theta-1)^2}+\frac{2-\eps'}{\Theta-1}+3+\frac{3}{\Theta-1}-4\eps'.
\end{equation*}
Now let the final request
\begin{equation*}
\sigma_3=\left(\frac{3}{(\Theta-1)^2}-\eps',\frac{3}{(\Theta-1)^2}-\eps';\frac{3}{(\Theta-1)^2}+\frac{2}{\Theta-1}\right)
\end{equation*}
appear.
By assumption, we have $\Theta<2$, which implies
\begin{equation*}
2+\frac{1}{\Theta-1}-2\eps'=\frac{2\Theta-1}{\Theta-1}-2\eps'\overset{\Theta<2}{<}\frac{3}{(\Theta-1)^2}-\eps',
\end{equation*}
i.e., the position of the request $\sigma_3$ lies to the right of $\sigma_2^{(1)}$. Thus we have for all $t\ge\frac{3}{(\Theta-1)^2}+\frac{2-\eps'}{\Theta-1}+3+\frac{3}{\Theta-1}-4\eps'$ the equation
\begin{align*}
L(t,0,\{\sigma_1,\sigma_2^{(1)},\sigma_2^{(2)},\sigma_3\})&=\left|0-\left(-\frac{1}{\Theta-1}\right)\right|+\left|\left(-\frac{1}{\Theta-1}\right)-\frac{3}{(\Theta-1)^2}-\eps'\right|\\
&=\frac{2}{\Theta-1}+\frac{3}{(\Theta-1)^2}-\eps'.
\end{align*}
Therefore the final schedule is not started before time
\begin{equation*}
\frac{L(t,0,\{\sigma_1,\sigma_2^{(1)},\sigma_2^{(2)},\sigma_3\})}{\Theta-1}=\frac{2}{(\Theta-1)^2}+\frac{3}{(\Theta-1)^3}-\frac{\eps'}{\Theta-1}.
\end{equation*}
However, by assumption, we have $\Theta\ge\frac{1}{2}\left(1+\sqrt{5}\right)$ and $\eps<\frac{1}{4}(\frac{5\Theta^2-9\Theta+4}{2\Theta+1})$, i.e.,~$\eps'<\frac{1}{4}$, which implies
\begin{align*}
\timeSch_2+L(\timeSch_2,\posSch_2,\{\sigma_2^{(1)},\sigma_2^{(2)}\})&\myoverset{}{=}{55}\frac{3}{(\Theta-1)^2}+\frac{2-\eps'}{\Theta-1}+3+\frac{3}{\Theta-1}-4\eps'\\
&\myoverset{}{=}{55}\frac{3\Theta}{(\Theta-1)^2}+\frac{2\Theta}{\Theta-1}+1-4\eps'-\frac{\eps'}{\Theta-1}\\
&\myoverset{}{=}{55}\frac{3(\Theta(\Theta-1))}{(\Theta-1)^3}+\frac{2(\Theta(\Theta-1))}{(\Theta-1)^2}\\
&\qquad\qquad\qquad+1-4\eps'-\frac{\eps'}{\Theta-1}\\
&\myoverset{$\eps'<\frac{1}{4}$}{>}{55}\frac{3(\Theta(\Theta-1))}{(\Theta-1)^3}+\frac{2(\Theta(\Theta-1))}{(\Theta-1)^2}-\frac{\eps'}{\Theta-1}\\
&\myoverset{$\Theta\ge\frac{1}{2}\left(1+\sqrt{5}\right)$}{\ge}{55}\frac{3}{(\Theta-1)^3}+\frac{2}{(\Theta-1)^2}-\frac{\eps'}{\Theta-1}\\
&\myoverset{}{=}{55}\frac{L(t,0,\{\sigma_1,\sigma_2^{(1)},\sigma_2^{(2)},\sigma_3\})}{\Theta-1}
\end{align*}
i.e., the starting time of the schedule $S_3$ is the ending time of the schedule $S_2$ and we have
\begin{equation*}
\timeSch_3=\frac{3}{(\Theta-1)^2}+\frac{2-\eps'}{\Theta-1}+3+\frac{3}{\Theta-1}-4\eps'.
\end{equation*}
The schedule $S_3$ needs time
\begin{equation*}
L(\timeSch_3,\posSch_3,\{\sigma_3\})=\left|\left(-\frac{1}{\Theta-1}\right)-\left(\frac{3}{(\Theta-1)^2}-\eps'\right)\right|=\frac{1}{\Theta-1}+\frac{3}{(\Theta-1)^2}-\eps'
\end{equation*}
To sum it up, we have
\begin{align*}
\algoS&=\timeSch_3+L(\timeSch_3,\posSch_3,\{\sigma_3\})\\
&=\frac{3}{(\Theta-1)^2}+\frac{2-\eps'}{\Theta-1}+3+\frac{3}{\Theta-1}-4\eps'\\
&\qquad+\frac{1}{\Theta-1}+\frac{3}{(\Theta-1)^2}-\eps'\\
&=\frac{6}{(\Theta-1)^2}+\frac{6}{\Theta-1}+3-\frac{5\Theta-4}{\Theta-1}\eps'.
\end{align*}
On the other hand, \opt goes from the origin straight to position $-\frac{1}{\Theta-1}$ serving request $\sigma_2^{(2)}$ at time $\frac{1}{\Theta-1}+\eps'$ (i.e., it has to wait for $\eps'$ units of time after it reaches position $-\frac{1}{\Theta-1}$). Then \opt walks straight from the origin to position $\frac{3}{(\Theta-1)^2}-\eps'$ serving all remaining requests. Thus, we have
\begin{align*}
\optS&=\left|0-\left(-\frac{1}{\Theta-1}\right)\right|+\eps'+\left|-\frac{1}{\Theta-1}-\left(\frac{3}{(\Theta-1)^2}-\eps'\right)\right|\\
&=\frac{2}{\Theta-1}+\frac{3}{(\Theta-1)^2}.
\end{align*}
Note that \opt can do this even if $c=1$ since for all requests the starting point is equal to the ending point. Since we have $\eps'=\frac{2\Theta+1}{5\Theta^2-9\Theta+4}\eps$, we finally obtain
\begin{align*}
\frac{\algoS}{\optS}&=\frac{\frac{6}{(\Theta-1)^2}+\frac{6}{\Theta-1}+3-\frac{5\Theta-4}{\Theta-1}\eps'}{\frac{2}{\Theta-1}+\frac{3}{(\Theta-1)^2}}\\
&=\frac{3\Theta^2+3}{2\Theta+1}-\frac{5\Theta^2-9\Theta+4}{2\Theta+1}\eps'\\
&=\frac{3\Theta^2+3}{2\Theta+1}-\eps,
\end{align*}
as claimed.
\end{proof}

Recall that the optimal parameter $\Theta^*$ established in \Cref{theorem: General Upper Bound} is the only positive, real solution of the equation
\begin{equation*}
\frac{3\Theta^2+3}{2\Theta+1}=\frac{2\Theta^2-\Theta+1}{\Theta^2-1},
\end{equation*}
which is $\Theta^*\approx 1.7125$. Therefore, according to \Cref{proposition: Lower Bound Waiting} and \Cref{proposition: Lower Bound No Waiting} the parameter~$\Theta^*$ lies in the range where the upper bounds of Propositions \ref{proposition: Upper Bound Waiting} and \ref{proposition: Upper Bound No Waiting} are both tight. It remains to make sure that for all~$\Theta$ that lie outside of this range the competitive ratio of $\algo_\Theta$ is larger than~$\rho^*\approx 2.6662$.
Let $\eps>0$ with $\eps<\frac{4\Theta+4}{\Theta-1}\cdot\min\{\frac{\Theta}{2\Theta-2},\frac{\Theta^2-\Theta-2}{(\Theta-1)^2},\frac{1}{\Theta-1}\}$ (note that $\frac{\Theta^2-\Theta-2}{(\Theta-1)^2}>0$ for $\Theta>2$) and $\eps'=\frac{\Theta-1}{4\Theta+4}\eps$.
Consider the set of requests $\sigma_{\Theta>2}=\{\sigma_1,\sigma_2^{(1)},\sigma_2^{(2)},\sigma_3\}$ with
\begin{align*}
\sigma_1&:=(1,1;0),\\
\sigma_2^{(1)}&:=\left(\frac{\Theta-2}{2\Theta-2}+\eps',1;\frac{1}{\Theta-1}+\eps'\right),\\
\sigma_2^{(2)}&:=\left(-\frac{1}{\Theta-1}+\eps',-\frac{1}{\Theta-1}+\eps';\frac{1}{\Theta-1}+\eps'\right),\\
\sigma_3&:=\left(1,1;\frac{\Theta+1}{(\Theta-1)^2}+\eps'\right).\\
\end{align*}
We compute \algo's completion time for the set of requests $\sigma_{\Theta>2}$ in the case $2<\Theta\le 1+\sqrt{2}$ and in the case $\Theta> 1+\sqrt{2}$.

\begin{lemma}\label{lemma: Lower Bound One}
Let the capacity $c\in\N\cup\{\infty\}$ of the server be arbitrary but fixed and let $2<\Theta\le 1+\sqrt{2}$. We have 
\begin{equation*}
\frac{\algo(\sigma_{\Theta>2})}{\opt(\sigma_{\Theta>2})}\ge\frac{3\Theta^2-2\Theta+1}{\Theta^2-1}-\eps.
\end{equation*}
In particular, we have 
\begin{equation*}
\frac{\algo(\sigma_{\Theta>2})}{\opt(\sigma_{\Theta>2})}>\rho^*\approx 2.6662.
\end{equation*}
for $\Theta\in (2,1+\sqrt{2}]\approx(2,2.4142]$ and sufficiently small $\eps$.
\end{lemma}
\begin{proof}
For all $t\ge 0$, we have $L(t,0,\{\sigma_1\})=1$. Thus, \algo starts its first schedule $S_1$ at time $\timeSch_1=\frac{1}{\Theta-1}$ and reaches position $\posSch_2=1$ at time $\frac{\Theta}{\Theta-1}$. 
We have $\eps<\frac{4\Theta+4}{\Theta-1}\cdot\frac{\Theta}{2\Theta-2}$, i.e., $\eps'<\frac{\Theta}{2\Theta-2}$, which implies
\begin{equation*}
0<\frac{\Theta-2}{2\Theta-2}+\eps'\overset{\eps'<\frac{\Theta}{2\Theta-2}}{<}1
\end{equation*}
for $\Theta>2$, i.e. the starting position of $\sigma_2^{(1)}$ is between $0$ and $1$. For $t\ge\frac{\Theta}{\Theta-1}$ we have 
\begin{align*}
L(t,0,\{\sigma_1,\sigma_2^{(1)},\sigma_2^{(2)}\})&=\left|0-\left(-\frac{1}{\Theta-1}+\eps'\right)\right|+2\eps'+\left|\left(-\frac{1}{\Theta-1}+\eps'\right)-1\right|\\
&=\frac{2}{\Theta-1}+1.
\end{align*}
Thus, the second schedule $S_2$ is not started before time 
\begin{equation*}
\frac{L\left(\frac{\Theta}{\Theta-1},0,\{\sigma_1,\sigma_2^{(1)},\sigma_2^{(2)}\}\right)}{\Theta-1}=\frac{2}{(\Theta-1)^2}+\frac{1}{\Theta-1}=\frac{\Theta+1}{(\Theta-1)^2}.
\end{equation*}
By assumption, we have $\Theta\le 1+\sqrt{2}$, which implies that for the time $\frac{\Theta}{\Theta-1}$, when \algo reaches position $\posSch_2=1$, the inequality 
\begin{equation*}
\frac{L\left(\frac{\Theta}{\Theta-1},0,\{\sigma_1,\sigma_2^{(1)},\sigma_2^{(2)}\}\right)}{\Theta-1}=\frac{\Theta+1}{(\Theta-1)^2}\ge \frac{\Theta}{\Theta-1}
\end{equation*}
holds.
Thus, \algo has a waiting period and starts the schedule $S_2$ at time
\begin{equation*}
\timeSch_2=\frac{L\left(\frac{\Theta}{\Theta-1},0,\{\sigma_1,\sigma_2^{(1)},\sigma_2^{(2)}\}\right)}{\Theta-1}=\frac{\Theta+1}{(\Theta-1)^2}
\end{equation*}
If \algo serves $\sigma_2^{(2)}$ before serving $\sigma_2^{(1)}$ the time it needs is at least
\begin{equation*}
\left|1-\left(-\frac{1}{\Theta-1}+\eps'\right)\right|+\left|\left(-\frac{1}{\Theta-1}+\eps'\right)-1\right|=\frac{2\Theta}{\Theta-1}-2\eps'.
\end{equation*}
The best schedule that serves $\sigma_2^{(2)}$ after serving $\sigma_2^{(1)}$ needs time
\begin{align*}
&\phantom{=}\;\;\left|1-\left(\frac{\Theta-2}{2\Theta-2}+\eps'\right)\right|+\left|\left(\frac{\Theta-2}{2\Theta-2}+\eps'\right)-1\right|+\left|1-\left(-\frac{1}{\Theta-1}+\eps'\right)\right|\\
&=\frac{\Theta}{2\Theta-2}-\eps'+\frac{\Theta}{2\Theta-2}-\eps'+\frac{\Theta}{\Theta-1}-\eps'\\
&=\frac{2\Theta}{\Theta-1}-3\eps'
\end{align*}
Thus, \algo serves $\sigma_2^{(2)}$ after serving $\sigma_2^{(1)}$ and finishes $S_2$ at position $\posSch_3=-\frac{1}{\Theta-1}+\eps'$ at time
\begin{align*}
\timeSch_2+L(\timeSch_2,\posSch_2,\{\sigma_2^{(1)},\sigma_2^{(2)}\})=\frac{\Theta+1}{(\Theta-1)^2}+\frac{2\Theta}{\Theta-1}-3\eps'=\frac{2\Theta^2-\Theta+1}{(\Theta-1)^2}-3\eps'.
\end{align*}
We have for all $t\ge\frac{2\Theta^2-\Theta+1}{(\Theta-1)^2}-3\eps'$ the equation
\begin{align*}
L(t,0,\{\sigma_1,\sigma_2^{(1)},\sigma_2^{(2)},\sigma_3\})&=\left|0-\left(-\frac{1}{\Theta-1}\right)\right|+\left|\left(-\frac{1}{\Theta-1}\right)-1\right|\\
&=\frac{2}{\Theta-1}+1.
\end{align*}
Therefore the final schedule is not started before time
\begin{equation*}
\frac{L(t,0,\{\sigma_1,\sigma_2^{(1)},\sigma_2^{(2)},\sigma_3\})}{\Theta-1}=\frac{2}{(\Theta-1)^2}+\frac{1}{\Theta-1}=\frac{\Theta+1}{(\Theta-1)^2},
\end{equation*}
which is equal to $\timeSch_2$ and thus smaller than $\timeSch_2+L(\timeSch_2,\posSch_2,\{\sigma_2^{(1)},\sigma_2^{(2)}\})$.
Therefore, the starting time of the schedule $S_3$ is the ending time of the schedule $S_2$ and we have
\begin{equation*}
\timeSch_3=\frac{2\Theta^2-\Theta+1}{(\Theta-1)^2}-3\eps'.
\end{equation*}
The schedule $S_3$ needs time
\begin{equation*}
L(\timeSch_3,\posSch_3,\{\sigma_3\})=\left|\left(-\frac{1}{\Theta-1}+\eps\right)-1\right|=\frac{1}{\Theta-1}+1-\eps'=\frac{\Theta}{\Theta-1}-\eps'
\end{equation*}
To sum it up, we have
\begin{align*}
\algoS&=\timeSch_3+L(\timeSch_3,\posSch_3,\{\sigma_3\})\\
&=\frac{2\Theta^2-\Theta+1}{(\Theta-1)^2}-3\eps'+\frac{\Theta}{\Theta-1}-\eps'\\
&=\frac{3\Theta^2-2\Theta+1}{(\Theta-1)^2}-4\eps'.
\end{align*}
On the other hand, \opt goes from the origin straight to position $-\frac{1}{\Theta-1}+\eps'$ serving request $\sigma_2^{(2)}$ at time $\frac{1}{\Theta-1}+\eps'$ (i.e., it has to wait for $2\eps'$ units of time after it reaches position $-\frac{1}{\Theta-1}$). Then \opt walks straight to position $\frac{\Theta-2}{2\Theta-2}+\eps'$ collecting the request $\sigma_2^{(1)}$. Note that the release time of $\sigma_2^{(2)}$ is the same as of $\sigma_2^{(1)}$ and thus \opt has no waiting time at position $\frac{\Theta-2}{2\Theta-2}+\eps'$. \opt reaches position $1$ and delivers $\sigma_2^{(1)}$ at time $\frac{2}{\Theta-1}+1=\frac{\Theta+1}{\Theta-1}$. By assumption we have $\Theta>2$ and $\eps<\frac{4\Theta+4}{\Theta-1}\cdot\frac{\Theta^2-\Theta-2}{(\Theta-1)^2}$, i.e., $\eps'<\frac{\Theta^2-\Theta-2}{(\Theta-1)^2}$, which implies
\begin{equation*}
\frac{\Theta+1}{\Theta-1}=\frac{\Theta+1}{(\Theta-1)^2}+\frac{\Theta^2-\Theta-2}{(\Theta-1)^2}>\frac{\Theta+1}{(\Theta-1)^2}+\eps'.
\end{equation*}
Thus, \opt has no waiting time at position $1$ and can serve the requests $\sigma_1$ and $\sigma_3$ at arrival.
To sum it up, we have
\begin{align*}
\optS&=\left|0-\left(-\frac{1}{\Theta-1}+\eps'\right)\right|+2\eps'+\left|\left(-\frac{1}{\Theta-1}+\eps'\right)-1\right|\\
&=\frac{2}{\Theta-1}+1=\frac{\Theta+1}{\Theta-1}.
\end{align*}
Note that \opt can do this even if $c=1$ since $\sigma_2^{(1)}$ is the only transportation request and no other request lies between its starting position and destination. Since we have $\eps'=\frac{\Theta-1}{4\Theta+4}\eps$, we finally obtain
\begin{align*}
\frac{\algoS}{\optS}&=\frac{\frac{3\Theta^2-2\Theta+1}{(\Theta-1)^2}-4\eps'}{\frac{\Theta+1}{\Theta-1}}\\
&=\frac{3\Theta^2-2\Theta+1}{\Theta^2-1}-\eps=:g_1-\eps.\\
\end{align*}
The function $g_1$ is monotonically decreasing on $(2,1+\sqrt{2}]$. Therefore, we have
\begin{equation*}
\frac{\algoS}{\optS}+\eps>g_1(1+\sqrt{2})=2\sqrt{2}>2.82>\rho^*
\end{equation*}
for all $\Theta\in (2,1+\sqrt{2}]$ and $\frac{\algoS}{\optS}>\rho^*$ for sufficiently small $\eps$.
\end{proof}

\begin{lemma}\label{lemma: Lower Bound Two}
Let the capacity $c\in\N\cup\{\infty\}$ of the server be arbitrary but fixed and let $\Theta> 1+\sqrt{2}$. We have 
\begin{equation*}
\frac{\algo(\sigma_{\Theta>2})}{\opt(\sigma_{\Theta>2})}\ge\frac{4\Theta}{\Theta+1}-\eps.
\end{equation*}
In particular, we have 
\begin{equation*}
\frac{\algo(\sigma_{\Theta>2})}{\opt(\sigma_{\Theta>2})}>\rho^*\approx 2.6662.
\end{equation*}
for $\Theta\in (1+\sqrt{2},\infty)\approx(2.4142,\infty)$ and sufficiently small $\eps$.
\end{lemma}
\begin{proof}
For all $t\ge 0$, we have $L(t,0,\{\sigma_1\})=1$. Thus, \algo starts its first schedule $S_1$ at time $\timeSch_1=\frac{1}{\Theta-1}$ and reaches position $\posSch_2=1$ at time $\frac{\Theta}{\Theta-1}$. 
We have $\eps<\frac{4\Theta+4}{\Theta-1}\cdot\frac{\Theta}{2\Theta-2}$, i.e., $\eps'<\frac{\Theta}{2\Theta-2}$, which implies
\begin{equation*}
0<\frac{\Theta-2}{2\Theta-2}+\eps'\overset{\eps'<\frac{\Theta}{2\Theta-2}}{<}1
\end{equation*}
for $\Theta>2$, i.e. the starting position of $\sigma_2^{(1)}$ is between $0$ and $1$. For $t\ge\frac{\Theta}{\Theta-1}$ we have 
\begin{align*}
L(t,0,\{\sigma_1,\sigma_2^{(1)},\sigma_2^{(2)}\})&=\left|0-\left(-\frac{1}{\Theta-1}+\eps'\right)\right|+2\eps'+\left|\left(-\frac{1}{\Theta-1}+\eps'\right)-1\right|\\
&=\frac{2}{\Theta-1}+1.
\end{align*}
Thus, the second schedule $S_2$ is not started before time 
\begin{equation*}
\frac{L\left(\frac{\Theta}{\Theta-1},0,\{\sigma_1,\sigma_2^{(1)},\sigma_2^{(2)}\}\right)}{\Theta-1}=\frac{2}{(\Theta-1)^2}+\frac{1}{\Theta-1}=\frac{\Theta+1}{(\Theta-1)^2}.
\end{equation*}
By assumption, we have $\Theta> 1+\sqrt{2}$, which implies that for the time $\frac{\Theta}{\Theta-1}$, when \algo reaches position $\posSch_2=1$, the inequality 
\begin{equation*}
\frac{L\left(\frac{\Theta}{\Theta-1},0,\{\sigma_1,\sigma_2^{(1)},\sigma_2^{(2)}\}\right)}{\Theta-1}=\frac{\Theta+1}{(\Theta-1)^2}< \frac{\Theta}{\Theta-1}
\end{equation*}
holds.
Thus, \algo has no waiting period and the starting time of the schedule $S_2$ is the ending time of the schedule $S_1$. We have
\begin{equation*}
\timeSch_2\frac{\Theta}{\Theta-1}
\end{equation*}
If \algo serves $\sigma_2^{(2)}$ before serving $\sigma_2^{(1)}$ the time it needs is at least
\begin{equation*}
\left|1-\left(-\frac{1}{\Theta-1}+\eps'\right)\right|+\left|\left(-\frac{1}{\Theta-1}+\eps'\right)-1\right|=\frac{2\Theta}{\Theta-1}-2\eps'.
\end{equation*}
The best schedule that serves $\sigma_2^{(2)}$ after serving $\sigma_2^{(1)}$ needs time
\begin{align*}
&\phantom{=}\;\;\left|1-\left(\frac{\Theta-2}{2\Theta-2}+\eps'\right)\right|+\left|\left(\frac{\Theta-2}{2\Theta-2}+\eps'\right)-1\right|+\left|1-\left(-\frac{1}{\Theta-1}+\eps'\right)\right|\\
&=\frac{\Theta}{2\Theta-2}-\eps'+\frac{\Theta}{2\Theta-2}-\eps'+\frac{\Theta}{\Theta-1}-\eps'\\
&=\frac{2\Theta}{\Theta-1}-3\eps'
\end{align*}
Thus, \algo serves $\sigma_2^{(2)}$ after serving $\sigma_2^{(1)}$ and finishes $S_2$ at position $\posSch_3=-\frac{1}{\Theta-1}+\eps'$ at time
\begin{align*}
\timeSch_2+L(\timeSch_2,\posSch_2,\{\sigma_2^{(1)},\sigma_2^{(2)}\})=\frac{\Theta}{\Theta-1}+\frac{2\Theta}{\Theta-1}-3\eps'=\frac{3\Theta}{\Theta-1}-3\eps'.
\end{align*}
We have for all $t\ge\frac{3\Theta}{\Theta-1}-3\eps'$ the equation
\begin{align*}
L(t,0,\{\sigma_1,\sigma_2^{(1)},\sigma_2^{(2)},\sigma_3\})&=\left|0-\left(-\frac{1}{\Theta-1}\right)\right|+\left|\left(-\frac{1}{\Theta-1}\right)-1\right|\\
&=\frac{2}{\Theta-1}+1.
\end{align*}
Therefore the final schedule is not started before time
\begin{equation*}
\frac{L(t,0,\{\sigma_1,\sigma_2^{(1)},\sigma_2^{(2)},\sigma_3\})}{\Theta-1}=\frac{2}{(\Theta-1)^2}+\frac{1}{\Theta-1}=\frac{\Theta+1}{(\Theta-1)^2},
\end{equation*}
which is, as before, smaller than $\timeSch_2$ and thus smaller than $\timeSch_2+L(\timeSch_2,\posSch_2,\{\sigma_2^{(1)},\sigma_2^{(2)}\})$.
Therefore, the starting time of the schedule $S_3$ is the ending time of the schedule $S_2$ and we have
\begin{equation*}
\timeSch_3=\frac{3\Theta}{\Theta-1}-3\eps'.
\end{equation*}
The schedule $S_3$ needs time
\begin{equation*}
L(\timeSch_3,\posSch_3,\{\sigma_3\})=\left|\left(-\frac{1}{\Theta-1}+\eps\right)-1\right|=\frac{1}{\Theta-1}+1-\eps'=\frac{\Theta}{\Theta-1}-\eps'
\end{equation*}
To sum it up, we have
\begin{align*}
\algoS&=\timeSch_3+L(\timeSch_3,\posSch_3,\{\sigma_3\})\\
&=\frac{3\Theta}{\Theta-1}-3\eps'+\frac{\Theta}{\Theta-1}-\eps'\\
&=\frac{4\Theta}{\Theta-1}-4\eps'.
\end{align*}
On the other hand, \opt goes from the origin straight to position $-\frac{1}{\Theta-1}+\eps'$ serving request $\sigma_2^{(2)}$ at time $\frac{1}{\Theta-1}+\eps'$ (i.e., it has to wait for $2\eps'$ units of time after it reaches position $-\frac{1}{\Theta-1}$). Then \opt walks straight to position $\frac{\Theta-2}{2\Theta-2}+\eps'$ collecting the request $\sigma_2^{(1)}$. Note that the release time of $\sigma_2^{(2)}$ is the same as of $\sigma_2^{(1)}$ and thus \opt has no waiting time at position $\frac{\Theta-2}{2\Theta-2}+\eps'$. \opt reaches position $1$ and delivers $\sigma_2^{(1)}$ at time $\frac{2}{\Theta-1}+1=\frac{\Theta+1}{\Theta-1}$. By assumption we have $\Theta>2$ and $\eps<\frac{4\Theta+4}{\Theta-1}\cdot\frac{1}{\Theta-1}$, i.e., $\eps'<\frac{1}{\Theta-1}$, which implies
\begin{equation*}
\frac{\Theta+1}{\Theta-1}=\frac{\Theta}{\Theta-1}+\frac{1}{\Theta-1}>\frac{\Theta}{\Theta-1}+\eps'.
\end{equation*}
Thus, \opt has no waiting time at position $1$ and can serve the requests $\sigma_1$ and $\sigma_3$ at arrival.
To sum it up, we have
\begin{align*}
\optS&=\left|0-\left(-\frac{1}{\Theta-1}+\eps'\right)\right|+2\eps'+\left|\left(-\frac{1}{\Theta-1}+\eps'\right)-1\right|\\
&=\frac{2}{\Theta-1}+1=\frac{\Theta+1}{\Theta-1}.
\end{align*}
Note that \opt can do this even if $c=1$ since $\sigma_2^{(1)}$ is the only transportation request and no other request lies between its starting position and destination. Since we have $\eps'=\frac{\Theta-1}{4\Theta+4}\eps$, we finally obtain
\begin{align*}
\frac{\algoS}{\optS}&=\frac{\frac{4\Theta}{\Theta-1}-4\eps'}{\frac{\Theta+1}{\Theta-1}}\\
&=\frac{4\Theta}{\Theta+1}-\eps=:g_2-\eps.\\
\end{align*}
The function $g_1$ is monotonically increasing on $(1+\sqrt{2},\infty)$. Therefore, we have
\begin{equation*}
\frac{\algoS}{\optS}+\eps>g_2(1+\sqrt{2})=2\sqrt{2}>2.82>\rho^*
\end{equation*}
for all $\Theta\in (1+\sqrt{2},\infty)$ and $\frac{\algoS}{\optS}>\rho^*$ for sufficiently small $\eps$.
\end{proof}

\begin{lemma}\label{lemma: Remaining Lower Bounds}
Let $\Theta>2$. There is a set of requests $\sigma_{\Theta>2}$ such that
\begin{equation*}
\frac{\algo(\sigma_{\Theta>2})}{\opt(\sigma_{\Theta>2})}>\rho^*\approx 2.6662.
\end{equation*}
\end{lemma}
\begin{proof}
This is immediate consequence of \Cref{lemma: Lower Bound One} and \Cref{lemma: Lower Bound Two}.
\end{proof}

\Cref{figure: Bounds} shows the upper and lower bounds that we have established. \Cref{theorem: Main Theorem} now follows from \Cref{theorem: General Upper Bound} combined with Propositions~\ref{proposition: Lower Bound Waiting} and~\ref{proposition: Lower Bound No Waiting}, as well as \Cref{lemma: Remaining Lower Bounds}.

\begin{figure}
\centering\begin{tikzpicture}
\begin{axis}[
samples=100,
xmin=1,
xmax=4,
xtick={1,2,...,4},
ytick={2,3,4},
ymin=2,
ymax=4,
xlabel=scaling paramter $\Theta$,
ylabel=competitive ratio $\rho$,
height=8cm,
width=12cm,
ylabel near ticks,
xlabel near ticks,
]

\addplot[no markers, thick, domain=1.01:2][black!40!green,line width=1pt]{(2*x^2-x+1)/(x^2-1)};
\addplot[no markers, thick, domain=2:6][black!40!green,line width=1pt, dashed]{(2*x^2-x+1)/(x^2-1)};

\addplot[no markers, domain=1:1.6180339887][black!40!red,line width=1pt, dashed]{(3*x^2+3)/(2*x+1)};
\addplot[no markers, domain=1.6180339887:2][black!40!red,line width=1pt]{(3*x^2+3)/(2*x+1)};
\addplot[no markers, domain=2:2.41421356237][name path=A1,black!40!red,line width=1pt, dashed]{(3*x^2+3)/(2*x+1)};
\addplot[no markers, domain=2.41421356237:6][name path=A2,black!40!red,line width=1pt, dashed]{(3*x^2+3)/(2*x+1)};

\addplot[no markers, domain=2:2.41421356237][name path=B1,black!40!blue,line width=1pt]{(3*x^2-2*x+1)/(x^2-1)};
\addplot[no markers, domain=2.41421356237:6][name path=B2,black!40!blue,line width=1pt]{(4*x)/(x+1)};

\draw[densely dashed,line width=0.5pt,black!70] (axis cs:1.71249,\pgfkeysvalueof{/pgfplots/ymin}) -- (axis cs:1.71249,\pgfkeysvalueof{/pgfplots/ymax});
\addplot[no markers,densely dashed, domain=1:6,line width=0.5pt, black!70]{2.6662};

\addplot[black!6!red!15, postaction={pattern=my north east lines,pattern color=gray}] fill between[of=A1 and B1];
\addplot[black!6!red!15, postaction={pattern=my north east lines,pattern color=gray}] fill between[of=A2 and B2];

\node[right] at (axis cs: 1.5,3.3) {$f_1$};
\node[left] at (axis cs: 2.2,3.3) {$f_2$};
\node[right] at (axis cs: 2.1,2.8) {$g_1$};
\node[left] at (axis cs: 3.1,2.9) {$g_2$};
\node[above] at (axis cs: 1.97,2.05) {$\Theta^*\approx 1.71$};
\node[below right] at (axis cs: 1,2.67) {$\rho^*\approx 2.67$};
\end{axis}
\end{tikzpicture}
\caption{Overview of our bounds for \algo. 
The functions~$f_1$ (green) / $f_2$ (red) are upper bounds for the cases where \algo waits / does not wait before starting the final schedule, respectively.
The upper bounds are drawn solid in the domains where they are tight for their corresponding case.
The functions $g_1$ and $g_2$ (blue) are general lower bounds.}
\label{figure: Bounds}
\end{figure}
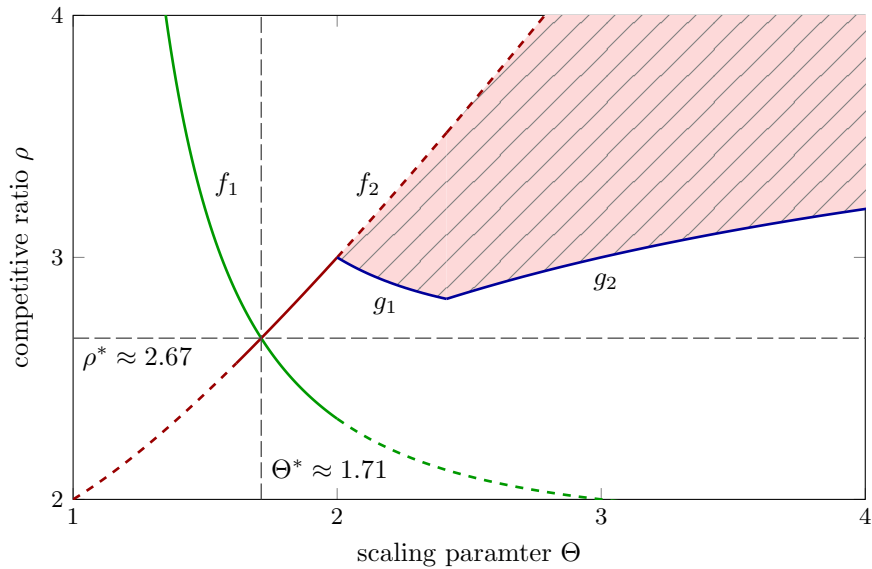

\begin{proof}[Proof of \ref{theorem: Main Theorem}]
We have shown in \Cref{proposition: Lower Bound Waiting} that the upper bound 
\begin{equation*}
\frac{\algoS}{\optS}\le f_1(\Theta) = \frac{2\Theta^2-\Theta+1}{\Theta^2-1}
\end{equation*}
established in \Cref{proposition: Upper Bound Waiting} for the case, where \algo waits before starting the final schedule, is tight for all $\Theta\in(1,2)$. Furthermore, we have shown in \Cref{proposition: Lower Bound No Waiting} that the upper bound
\begin{equation*}
\frac{\algoS}{\optS}\le f_2(\Theta) = \frac{3\Theta^2+3}{2\Theta+1}
\end{equation*}
established in \Cref{proposition: Upper Bound No Waiting} for the case, where \algo does not wait before starting the final schedule, is tight for all $\Theta\in(\frac{1}{2}(1+\sqrt{5}),2]$. Since $\Theta^*\approx 1.71249$ lies in those ranges, the competitive ratio of $\algo_{\Theta^*}$ is indeed exactly~$\rho^*$. 

It remains to show that for every~$\Theta>1$ with~$\Theta\neq\Theta^*$ the competitive ratio is larger. First, according to \Cref{lemma: Remaining Lower Bounds}, the competitive ratio of \algo with parameter $\Theta\in (2,\infty)$ is larger than $\rho^*$. By monotonicity of~$f_1$, every function value in $(1,\Theta^*)$ is larger than $f_1(\Theta^*)=\rho^*$. Thus, the competitive ratio of \algo with parameter $\Theta\in (1,\Theta^*)$ is larger than $\rho^*$, since $f_1$ is tight on $(1,\Theta^*)$ by \Cref{proposition: Lower Bound Waiting}. Similarly, by monotonicity of~$f_2$, every function value in $(\Theta^*,2]$ is larger than~$f_2(\Theta^*)=\rho^*$. Thus, the competitive ratio of \algo with parameter~$\Theta\in (\Theta^*,2]$ is larger than~$\rho^*$, since $f_2$ is tight on $(\Theta^*,2]$ by \Cref{proposition: Lower Bound No Waiting}.
\end{proof}

\newpage
\bibliography{LowerBounds}
\addcontentsline{toc}{chapter}{Bibliography}
\newpage
\appendix

\section{Proof of Lemma \ref{lemma: Critical Requests}}\label{appendix: Proof Of Critical Requests}

In this section we prove \Cref{lemma: Critical Requests}. 
The proof is almost identical to the proof of \cite[Lemma 6]{Disser1}. Since there are however several parts where inequalities change slightly, we decided to present the full proof here.
\lemmaCriticalRequests*

Let the requests $\SL{}$ and $\SR{}$ be critical. Furthermore, let $\TA\in\{\TL{},\TR{}\}$ be the starting position of the request $\SA\in\{\SL{},\SR{}\}$ that is served first by \alg and let $\TB\in\{\TL{},\TR{}\}$ be the starting position of the request $\SB\in\{\SL{},\SR{}\}$ that is not served first by \alg. By properties (iii) and (iv) of \Cref{definition: Critical Requests}, \alg cannot serve $\SA$ before time $(2\rh-2)|\TB|+(\rh-2)|\TA|$.
Thus, we have 
\begin{equation}
\alg(\sigma_\rh)\ge (2\rh-2)\TB+(\rh-2)\TA+|\TA-\TB|=(2\rh-1)|\TB|+(\rh-1)|\TA|.\label{equation: Alg Is Fast}
\end{equation}
We have equality in inequality (\ref{equation: Alg Is Fast}) if \alg serves $\SA$ the earliest possible time and then moves directly to position $\TB$. However, in general \alg does not need to do this and instead can wait. 
At time $t\ge\max\{|\TA|,|\TB|\}$, we have
\begin{equation*}
\alg(\sigma_\rh)\ge t+|\pos{t}-\TA|+|\TA-\TB|
\end{equation*}
if \alg still has to serve $\SA$ and 
\begin{equation*}
\alg(\sigma_\rh)\ge t+|\pos{t}-\TB|
\end{equation*}
if $\SA$ is served and only $\SB$ is left to be served. 
We want to measure the delay of \alg at a time $t\ge\max\{|\TA|,|\TB|\}$, i.e. the difference between the time \alg needs at least to serve both requests $\SA$ and $\SB$ and the time $(2\rh-1)|\TB|+(\rh-1)|\TA|$.
We define for $t\ge\max\{|\TA|,|\TB|\}$ the function
\begin{equation*}
\delay(t):=\begin{cases}
t+|\pos{t}-\TA|-(\rh-2)|\TA|-(2\rh-2)|\TB| & \,\text{if $\SA$ is not} \\
& \,\text{served at $t$,} \\
t+|\pos{t}-\TB|-(\rh-1)|\TA|-(2\rh-1)|\TB| & \,\text{if $\SA$ is served at $t$,} \\
& \,\text{but $\SB$ not,} \\
\text{undefined} & \,\text{otherwise.} \\
\end{cases}
\end{equation*}
We make the following observation about $\delay$.
\begin{observation}\label{observation: Delay}
Let $t\ge \max\{|\TA|,|\TB|\}$ be a time at which $\SB$ is not served yet. The earliest time \alg can serve $\SB$ is $(2\rh-1)|\TB|+(\rh-1)|\TA|+\delay(t)$.
\end{observation}

\begin{lemma}\label{lemma: Existance W}
There is a $W\ge 0$ with
\begin{equation*}
\delay\left(2|\TB{}|+|\TA{}|+\frac{W}{\rh-1}\right)=W
\end{equation*}
\end{lemma}
\begin{proof}
Because of property (ii) of \Cref{definition: Critical Requests}, at time $\max\{|\TA|,|\TB|\}$ neither $\SA$ nor $\SB$ has been served by \alg yet.
Since \alg serves $\SB$ after $\SA$, the request $\SB$ is not served before time 
\begin{equation*}
\max\{|\TA|,|\TB|\}+|\TA|+|\TB|\ge 2|\TB{}|+|\TA{}|,
\end{equation*}
i.e, $\delay(2\TB{}+\TA{})$ is defined.
Because of properties (iii) and (iv) of \Cref{definition: Critical Requests}, $\SA$ is not served before time $(2\rh-2)|\TB|+(\rh-2)|\TA|$.
Thus, for $t\ge (2\rh-2)\TB+(\rh-2)\TA$, we have $\delay(t)\ge 0$.
We have 
\begin{align}
2\TB{}+\TA{}&\myoverset{Def \ref{definition: Critical Requests} (v)}{\ge}{40} 2\TB{}+(3-\rh)\frac{-8\rh^2+50\rh-66}{4\rh^2-30\rh+50}|\TB|+(\rh-2)|\TA|\nonumber\\
&\myoverset{$2<\rh<2.5$}{>}{40}(2\rh-2)|\TB|+(\rh-2)|\TA|,\label{equation: Delay Bigger Zero}
\end{align}
i.e. $\delay(2\TB{}+\TA{})\ge 0$.
If $\delay(2\TB{}+\TA{})=0$, we have $W=0$ and are done. 
Otherwise, by inequality (\ref{equation: Delay Bigger Zero}), we have $\delay(2\TB{}+\TA{})>0$.
Note that \alg needs to serve $\SB$ at some point to be $(\rho-\eps)$-competitive.
Let $W^*$ be chosen such that \alg serves $\SB$ at time $2|\TB{}|+|\TA{}|$+$\frac{W^*}{\rh-1}$.
Therefore
\begin{equation*}
\delay\left(2|\TB{}|+|\TA{}|+\frac{W^*}{\rh-1}-\eps'\right)
\end{equation*}
is defined for some sufficiently small $\eps'\le |\TB|$.
Define the function 
\begin{equation*}
f(W):=\delay\left(2|\TB{}|+|\TA{}|+\frac{W}{\rh-1}\right)-W.
\end{equation*}
Note that $f$ is continuous and we have $f(0)>0$.
If 
\begin{equation*}
\delay\left(2|\TB{}|+|\TA{}|+\frac{W^*}{\rh-1}-\eps'\right)\le \frac{W^*}{\rh-1}-\eps'\overset{\rh>1}{<}W^*-(\rh-1)\eps',
\end{equation*}
we have $f(W^*-(\rh-1)\eps')<0$ and we find $W$ in the interval $(0,W^*-(\rh-1)\eps']$.
Otherwise, we have 
\begin{equation*}
\delay\left(2|\TB{}|+|\TA{}|+\frac{W^*}{\rh-1}-\eps'\right)> \frac{W^*}{\rh-1}-\eps'.
\end{equation*}
By \Cref{observation: Delay} \alg has not served $\SB$ at time 
\begin{equation*}
(2\rh-1)|\TB|+(\rh-1)|\TA|+\frac{W^*}{\rh-1}-\eps'\overset{\rh>2,\eps'\le |\TB|}{>}2|\TB|+|\TA|+\frac{W^*}{\rh-1}.
\end{equation*}
This is a contradiction to the fact, that $W^*$ was chosen such that \alg serves $\SB$ at time $2|\TB{}|+|\TA{}|$+$\frac{W^*}{\rh-1}$.
\end{proof}

\begin{lemma}\label{lemma: Serving Time Of SA}
Let $W\ge 0$ with
\begin{equation*}
\delay\left(2|\TB{}|+|\TA{}|+\frac{W}{\rh-1}\right)=W.
\end{equation*}
\alg serves $\SA$ no later than time $2|\TB{}|+|\TA{}|+\frac{W}{\rh-1}$.
\end{lemma}
\begin{proof}
Assume we have
\begin{equation}
2|\TB{}|+|\TA{}|+\frac{W}{\rh-1}\ge (2\rh-2)|\TB|+(\rh-2)|\TA|+W.\label{equation: Assumption Serving Time Of SA}
\end{equation}
Then, by definition of $W$ and \Cref{observation: Delay}, \alg can serve $\SB$ at time
\begin{align}
&\phantom{=}\;\,(2\rh-1)|\TB|+(\rh-1)|\TA|+\delay\left(2|\TB{}|+|\TA{}|+\frac{W}{\rh-1}\right)\nonumber\\
&=(2\rh-1)|\TB|+(\rh-1)|\TA|+W.\label{equation: Time Of Alg Obs}
\end{align}
Because of inequality (\ref{equation: Assumption Serving Time Of SA}), this can only be the case if \alg serves $\SA$ no later than time 
\begin{align*}
(2\rh-1)|\TB|+(\rh-1)|\TA|+W-|\TB|-|\TA|&\myoverset{}{=}{15}(2\rh-2)|\TB|+(\rh-2)|\TA|+W\\
&\myoverset{(\ref{equation: Assumption Serving Time Of SA})}{\le}{15} 2|\TB{}|+|\TA{}|+\frac{W}{\rh-1}.
\end{align*}
Thus, it remains to show inequality (\ref{equation: Assumption Serving Time Of SA}).
Because of property (i) of \Cref{definition: Critical Requests} all requests can be served the tours $\move{\TA}\oplus\move{\TB}$ and $\move{\TB}\oplus\move{\TA}$.
By inequality \ref{equation: Time Of Alg Obs}, we have $\alg(\sigma_\rh)\ge (2\rh-1)|\TB|+(\rh-1)|\TA|+W$. Thus, if we have
\begin{equation*}
\alg(\sigma_\rh)\ge (2\rh-1)|\TB|+(\rh-1)|\TA|+W> (\rho-\eps)(2|\TB{}|+|\TA{}|)\ge (\rho-\eps)\opt(\sigma_\rh),
\end{equation*}
\alg is not $(\rho-\eps)$-competitive. Therefore, we may assume 
\begin{equation*}
(2\rh-1)|\TB|+(\rh-1)|\TA|+W\le (\rho-\eps)(2|\TB{}|+|\TA{}|),
\end{equation*}
and thus
\begin{align}
W&\le(\rh-\eps) (2|\TB{}|+|\TA{}|)-(2\rh-1)|\TB|-(\rh-1)|\TA|\nonumber\\
&= (1-2\eps)|\TB|+(1-\eps)|\TA|\nonumber\\
&< |\TB|+|\TA|.\label{equation: Bound On W Serve SA}
\end{align}
Inequality (\ref{equation: Assumption Serving Time Of SA}) now is equivalent to the inequality
\begin{align*}
&\myoverset{}{}{40}\frac{2|\TB{}|+|\TA{}|-((2\rh-2)|\TB|+(\rh-2)|\TA|)}{1-\frac{1}{\rh-1}}\\
&\myoverset{}{=}{40} \frac{(\rh-1)((4-2\rh)|\TB{}|+(3-\rh)|\TA{}|)}{\rh-2}\\
&\myoverset{}{=}{40} \frac{(\rh-1)(4-2\rh)}{\rh-2}|\TB{}|+\frac{(\rh-1)(3-\rh)}{\rh-2}|\TA{}|\\
&\myoverset{Def \ref{definition: Critical Requests} (v)}{\ge}{40} |\TA{}|+ (2-2\rh)|\TB{}|\\
&\qquad\qquad+\frac{(-\rh^2+3\rh-1)(-8\rh^2+50\rh-66)}{(\rh-2)(4\rh^2-30\rh+50)}|\TB{}|\\
&\myoverset{}{\ge}{40} |\TA{}|+ \frac{5 \rh^3 - 36 \rh^2 + 86 \rh - 67}{2 \rh^3 - 19 \rh^2 + 55 \rh - 50}|\TB{}|\\
&\myoverset{$2< \rh< 2.5$}{>}{40} |\TA{}|+ |\TB{}|\\
&\myoverset{(\ref{equation: Bound On W Serve SA})}{>}{40} W
\end{align*}
if we solve inequality (\ref{equation: Assumption Serving Time Of SA}) for $W$.
\end{proof}

Now we have all ingredients to proof \Cref{lemma: Critical Requests}.

\begin{proof}[Proof of \Cref{lemma: Critical Requests}]
Let $W\ge 0$ with
\begin{equation*}
\delay\left(2|\TB{}|+|\TA{}|+\frac{W}{\rh-1}\right)=W.
\end{equation*}
We present the request
\begin{align*}
\SA^+&\phantom{:}=(\TA^+,\TA^+;t_0^+)\\
&:=\left(\TA+\sgn{\TA}\frac{W}{\rh-1},\TA+\sgn{\TA}\frac{W}{\rh-1};2|\TB|+|\TA|+\frac{W}{\rh-1}\right)
\end{align*}
and distinguish two cases.

\paragraph*{\textit{Case 1: }At time $t_0^+$, {\normalfont\alg} is at least as close to $\TB$ as to $\TA^+$ or it serves $\SB$ before $\SA^+$.\\}
In this case, we do not present additional requests.
By \Cref{lemma: Serving Time Of SA}, \alg has served $\SA$ at time $t_0^+$ or before and by \Cref{observation: Delay} it does not serve $\SB$ earlier than time $(2\rh-1)|\TB|+(\rh-1)|\TA|+W$.
Thus, we have
\begin{align*}
\alg(\sigma_\rh)&\ge (2\rh-1)|\TB|+(\rh-1)|\TA|+W+|\TB|+|\TA|+\frac{W}{\rh-1}\\
&\ge \rh\left(2|\TB|+|\TA|+\frac{W}{\rh-1}\right)\\
&=\rh\opt(\sigma_\rh).
\end{align*}

\paragraph*{\textit{Case 2: }At time $t_0^+$, {\normalfont\alg} is closer to $\TA^+$ than to $\TB$ and it serves $\SA^+$ first.\\}
We assume that the offline server continues moving away from the origin after serving $\SA^+$ at time $\TA^+$. 
Then, the position of the  offline serve at time $t\ge |\TB|$ is $\sgn{\TA}t+2\TB$.
We denote by 
\begin{equation*}
M(t):=\frac{\sgn{\TA}t+3\TB}{2}
\end{equation*}
the midpoint between the current position of the offline server and the position $\TB$.
Note that the time $M^{-1}(p)$, when the midpoint is at position $p$ is given by
\begin{equation*}
M^{-1}(p):=|2p-3\TB|.
\end{equation*}
We again distinguish two cases

\paragraph*{\textit{Case 2.1: }{\normalfont\alg} does not serve $\SA^+$ until time $M^{-1}(\TA^+)$.\\}
In this case, we do not present additional requests.
Since we are in Case 2, neither $\SA^+$ nor $\SB$ is served at time $M^{-1}(\TA^+)$.
Thus, we have
\begin{align*}
\alg(\sigma_\rh)&\myoverset{}{\ge}{40} M^{-1}(\TA^+)+|\TA^+|+|\TB|\\
&\myoverset{}{=}{40}|2\TA^+-3\TB|+|\TA^+|+|\TB|\\
&\myoverset{}{=}{40}|2\TA+2\sgn{\TA}\frac{W}{\rh-1}-3\TB|+|\TA|+\frac{W}{\rh-1}+|\TB|\\
&\myoverset{}{=}{40} 3|\TA|+4|\TB|+3\frac{W}{\rh-1}\\
&\myoverset{$2< \rh< 2.5$}{>}{40} \rh|\TA|+2\rh|\TB|+3\frac{W}{\rh-1}\\
&\myoverset{}{>}{40} \rh\left(|\TA|+2|\TB|+\frac{W}{\rh-1}\right)\\
&\myoverset{}{=}{40} \rh\opt(\sigma_\rh).
\end{align*}

\paragraph*{\textit{Case 2.2: }{\normalfont\alg} serves $\SA^+$ before time $M^{-1}(\TA^+)$.\\}
By definition of $W$, the delay function is defined for time $\TA^+$, hence \alg has not served $\SB$ before time $\TA^+$.
Since \alg is to the right of the midpoint $M(\TA^+)$ at time $\TA^+$, there is a first time $t_\text{mid}$ at which $M(t_\text{mid})=\pos{t_\text{mid}}$.
We present the request
\begin{equation*}
\SA^{++}=(\TA^{++},\TA^{++};t_0^{++}):=(\sgn{\TA}t_\text{mid}+2\TB,\sgn{\TA}t_\text{mid}+2\TB;t_\text{mid}).
\end{equation*}
Note that \alg is at the midpoint between $\TA^{++}$ and $\TB$ and thus, both tours $\move{\TA^{++}}\oplus\move{\TB}$ and $\move{\TB}\oplus\move{\TA^{++}}$ incur identical costs for \alg.
We have
\begin{equation*}
\alg(\sigma_\rh)\ge t_\text{mid}+3\left(\frac{|\sgn{\TA}t_\text{mid}+2\TB-\TB|}{2}\right)=\frac{5t_\text{mid}+3|\TB|}{2}
\end{equation*}
We have $\opt(\sigma_\rh)=t_\text{mid}$, i.e., if we want to show
\begin{equation}
\alg(\sigma_\rh)\ge \frac{5t_\text{mid}+3|\TB|}{2}\ge \rh t_\text{mid}=\rh\opt(\sigma_\rh)\label{equation: Last Case Comp}
\end{equation}
Inequality (\ref{equation: Last Case Comp}) is equivalent to
\begin{equation}
(5-2\rh)t_\text{mid}\ge 3|\TB|.\label{equation: Last Case Comp Easier}
\end{equation}
Since $2\rh<2.5$, the coefficient $(5-2\rh)$ of $t_\text{mid}$ is positive. 
Thus we may assume $t_\text{mid}$ is minimal to show the inequality (\ref{equation: Last Case Comp Easier}).
By assumption, $\SA^+$ is already served at time $t_\text{mid}$. 
Hence, $t_\text{mid}$ is minimum if, starting at time $t_0^+$ at position $\pos{t_0^+}$, \alg serves $\SA^+$ and then moves towards the origin.
Then, $t_\text{mid}$ is the solution of the equation
\begin{equation}
\sgn{\TA}t_0^++|\pos{t_0^+}-\TA^+|+\TA^+-\sgn{\TA}t_\text{mid}=\frac{\sgn{\TA}t_\text{mid}+3\TB}{2}.\label{equation: Last Case TMID}
\end{equation}
Because of \Cref{lemma: Serving Time Of SA}, the request $\SA$ is already served at time $t_0^+$. 
Furthermore, since the position of $\SB$ has not been visited yet at time $t_0^+$, we have $\sgn{\TA}\pos{t_0^+}>\sgn{\TA}\TB$, i.e.,
\begin{equation*}
|\pos{t_0^+}-\TB|=\sgn{\TA}(\pos{t_0^+}-\TB)>0
\end{equation*}
and thus, because of $-\sgn{\TA}\TB=|\TB|$, we get
\begin{align}
\delay(t_0^+)&=t_0^++|\pos{t_0^+}-\TB|-(\rh-1)|\TA|-(2\rh-1)|\TB|\nonumber\\
&=t_0^++\sgn{\TA}\pos{t_0^+}-\sgn{\TA}\TB-(\rh-1)|\TA|-(2\rh-1)|\TB|\nonumber\\
&=t_0^++\sgn{\TA}\pos{t_0^+}+|\TB|-(\rh-1)|\TA|-(2\rh-1)|\TB|.\label{equation: Bound For Position}
\end{align}
Solving equation (\ref{equation: Bound For Position}) for $\sgn{\TA}\pos{t_0^+}$ gives
\begin{align}
&\myoverset{}{}{65}\sgn{\TA}\pos{t_0^+}\nonumber\\
&\myoverset{}{=}{65} \delay\left(2|\TB|+|\TA|+\frac{W}{\rh-1}\right)-\frac{W}{\rh-1}\nonumber\\
&\qquad\qquad\qquad\qquad+(\rh-2)|\TA|+(2\rh-4)|\TB|\nonumber\\
&\myoverset{}{=}{65} W-\frac{W}{\rh-1}+(\rh-2)|\TA|+(2\rh-4)|\TB|\nonumber\\
&\myoverset{}{=}{65} \frac{\rh-2}{\rh-1}W+(\rh-2)|\TA|+(2\rh-4)|\TB|\label{equation: Bound For Position Better}\\
&\myoverset{$\rh< 3$}{<}{65} \frac{W}{\rh-1}+(\rh-2)|\TA|+(2\rh-4)|\TB|\nonumber\\
&\myoverset{Def \ref{definition: Critical Requests} (v)}{\le}{65} \frac{W}{\rh-1}+\left((\rh-2)+(2\rh-4)\frac{4\rh^2-30\rh+50}{-8\rh^2+50\rh-66}\right)|\TA|\nonumber\\
&\myoverset{$1.9< \rh< 4.3$}{<}{65} \frac{W}{\rh-1}+|\TA|\nonumber\\
&\myoverset{}{=}{65} |\TA^+|\nonumber\\
&\myoverset{$\sgn{\TA^+}=\sgn{\TA}$}{=}{65} \sgn{\TA}\TA^+\nonumber.
\end{align}
Thus, we have 
\begin{equation}
|\pos{t_0^+}-\TA^+|=\sgn{\TA}(\TA^+-\pos{t_0^+})>0\label{equation: Bound For Position Useful}
\end{equation}
Using inequality (\ref{equation: Bound For Position Useful}) and plugging inequality (\ref{equation: Bound For Position Better}) into inequality (\ref{equation: Last Case TMID}) gives us
\begin{align}
\sgn{\TA}t_\text{mid}&\myoverset{}{=}{15} \frac{1}{3}(2\sgn{\TA}t_0^++2|\pos{t_0^+}-2\TA^+|+2\TA^+-3\TB)\nonumber\\
&\myoverset{(\ref{equation: Bound For Position Useful})}{=}{15} \frac{1}{3}(2\sgn{\TA}t_0^++2\sgn{\TA}\TA^+-2\sgn{\TA}\pos{t_0^+}+2\TA^+-3\TB)\nonumber\\
&\myoverset{}{=}{15} \frac{1}{3}\left(-7\TB+6\TA+\frac{(6\sgn{\TA})W}{\rh-1}-2\sgn{\TA}\pos{t_0^+}\right)\nonumber\\
&\myoverset{(\ref{equation: Bound For Position Better})}{=}{15} \frac{1}{3}\left(-(15-4\rh)\TB+(10-2\rh)\TA+\frac{(10-2\rh)\sgn{\TA}W}{\rh-1}\right)\label{equation: TMID Representation SGN}
\end{align}
Note that we also used $\sgn{\TA}=\sgn{\TA^+}=-\sgn{\TB}$.
Multiplying equality (\ref{equation: TMID Representation SGN}) with $\sgn{\TA}$ gives us
\begin{equation}
t_\text{mid}=\frac{1}{3}\left((15-4\rh)|\TB|+(10-2\rh)|\TA|+\frac{(10-2\rh)W}{\rh-1}\right).\label{equation: TMID Representation}
\end{equation}
By substituting (\ref{equation: TMID Representation}) into (\ref{equation: Last Case Comp Easier}) and noting that it is hardest to satisfy, when $W=0$, we get
\begin{equation*}
\frac{|\TA|}{|\TB|}\le \frac{4\rh^2-30\rh+50}{-8\rh^2+50\rh-66},
\end{equation*}
which is true due to \Cref{definition: Critical Requests} (v).
\end{proof}

\end{document}